\documentclass[11pt]{article}

\usepackage{latexsym}
\usepackage{amssymb}
\usepackage{amsthm}
\usepackage{amscd}
\usepackage{amsmath}
\usepackage{mathrsfs}
\usepackage{graphicx}
\usepackage{hyperref}
\usepackage{shuffle}
\usepackage{tikz-cd}
\usepackage[all]{xy}
\input xy \xyoption{frame}
\usepackage{ytableau}

\usepackage{enumerate,tikz}
\usepackage{color}
\usepackage{mathdots}
\usepackage[vcentermath,enableskew]{youngtab}

\usepackage{mathabx}
\usepackage{mathtools}
\usepackage{array}

\newcommand{\dK}[1]{\mathbin{\widesim{#1}{\mathsf{dK}}}}
\newcommand{\K}[1]{\mathbin{\widesim{#1}{\mathsf{K}}}}

\newcommand{\simRBS}[1]{\mathbin{\widesim{#1}{\rBS}}}

\def\cI{I}

\definecolor{darkred}{rgb}{0.7,0,0} % darkred color
\newcommand{\defn}[1]{{\color{darkred}\emph{#1}}} % emphasis of a definition

\numberwithin{equation}{section}
\theoremstyle{definition}
\newtheorem* {theorem*}{Theorem}
\newtheorem* {conjecture*}{Conjecture}
\newtheorem{theorem}{Theorem}[section]

\theoremstyle{definition}

\theoremstyle{definition}

\newtheorem* {example*}{Example}

\theoremstyle{definition}
\newtheorem{definition}[theorem]{Definition}
\theoremstyle{definition}

\newtheorem{proposition}[theorem]{Proposition}
\newtheorem{corollary}[theorem]{Corollary}

\newtheorem* {remark*}{Remark}

\theoremstyle{definition}
\newtheorem {example}[theorem]{Example}
\theoremstyle{definition}

\theoremstyle{definition}

\theoremstyle{definition}

\theoremstyle{definition}

\def\({\left(}
\def\){\right)}

\def\NN{\mathbb{N}}

\def\ZZ{\mathbb{Z}}

\newcommand{\cM}{\mathcal{M}}
\newcommand{\cN}{\mathcal{N}}

\newcommand{\sgn}{\mathrm{sgn}}

\def\barr{\begin{array}}
\def\earr{\end{array}}
\def\ba{\begin{aligned}}
\def\ea{\end{aligned}}
\def\be{\begin{equation}}
\def\ee{\end{equation}}

\def\qquand{\qquad\text{and}\qquad}

\def\cH{\mathcal H}

\def\PP{\mathbb{P}}

\def\ben{\begin{enumerate}}
\def\een{\end{enumerate}}

\def\cF{\mathcal F}

\def\fpf{{\tt {FPF}}}

\def\ellfpf{\ellhat_\fpf}

\newcommand{\cA}{\mathcal{A}}

\def\tS{\tilde S}

\def\Inv{\operatorname{Inv}}

\def\arcstart{\ \xy<0cm,-0.5cm>\xymatrix@R=.1cm@C=.5cm }
\newcommand{\arcstartc}[1]{\ \xy<0cm,-.15cm>\xymatrix@R=.1cm@C=#1cm}

\def\tp{\Theta^+}

\def\cF{\mathcal F}
\def\cFp{\cF_n^+}
\def\cFm{\cF_n^-}

\def\sfpf{\sgn_\fpf}
\def\m{\mathfrak{m}}
\def\n{\mathfrak{n}}

\def\invsim_i{\overset{\mathrm{i}}{\underset{\mathrm{inv}}{\sim}}}

\def\ellfpf{\ell^{\mathsf{FPF}}}

\def\m{\mathbf{m}}
\def\n{\mathbf{n}}

\def\lch{\mathrm{lch}}

\def\rBS{\mathsf{It}}

\newcommand{\widesim}[3][1.5]{
  \overset{#2}{\underset{#3}{\scalebox{#1}[1]{$\sim$}}}
}

\newcommand{\yttab}[1]{
\ytableausetup{boxsize = .6cm,aligntableaux=center}
 \begin{ytableau} #1 \end{ytableau}}

\def\Pafpf{{\operatorname{Par}_{\mathrm{FPF}}}}

\usepackage[colorinlistoftodos]{todonotes}

\newcommand{\dom}{\Omega_{\mathrm{dom}}}

\newcommand{\ol}[1]{\overline{#1}}
\newcommand{\RSYT}{\textup{RSYT}}
\newcommand{\offset}{\mathbf{s}}
\DeclareMathOperator{\rev}{\textup{rev}}

\usepackage{fullpage}
\usepackage{ytableau}
\usepackage{mathabx}
\usepackage{array}
\usepackage{colonequals}
\hypersetup{hidelinks,pdfauthor={Yifeng Zhang},
  pdftitle={Classification of molecules in an affine fixed-point-free W-graph}}

\graphicspath{{figures/}}

\begin{document}
\title{Classification of molecules\\in an affine fixed-point-free $W$-graph}
\author{Yifeng Zhang\\
School of Mathematical Sciences\\
South China Normal University\\
{\tt calvinz314159@gmail.com}}
\date{}

\maketitle
\vspace{-.6\baselineskip}

\begin{abstract}
In a $W$-graph on affine fixed-point-free involutions, the connected
components formed by bidirected edges are called molecules.
The affine matrix-ball construction represents each such involution
by a tabloid and an integer weight vector.
We classify the molecules by shape, weight, and tabloid charge modulo
the greatest common divisor of the column lengths.
Molecules of the same shape, even at different weights, have isomorphic
bidirected graphs up to a cyclic shift of the generator labels.
\end{abstract}

\setcounter{tocdepth}{2}
\begingroup
\small
\setlength{\baselineskip}{11pt}
\makeatletter
\renewcommand*\l@section[2]{\addpenalty{-\@highpenalty}\vspace{.24em}%
  \@dottedtocline{1}{0em}{1.5em}{\bfseries #1}{\bfseries #2}}
\makeatother
\tableofcontents
\endgroup
\clearpage

\section{Introduction}

For a Coxeter group $W$, a $W$-graph describes the action of its Hecke
algebra on a chosen basis. It is a weighted directed graph whose
vertices are labeled by sets of simple generators. Kazhdan and Lusztig
introduced these graphs in their construction of a basis for the
Hecke algebra itself \cite{KL}. The strongly connected components of
a $W$-graph are its \emph{cells}. Keeping just the edges present in
both directions gives the connected components called
\emph{molecules}. Stembridge gave local rules for the class of
$W$-graphs he called admissible \cite{Stem1,Stem2}. In finite type
$A$, Chmutov identified the bidirected parts of admissible molecules
with those of Kazhdan--Lusztig left cells \cite{ChmutovMol}, and Nguyen
proved the corresponding identification for admissible cells
\cite{NguyenCells}.

For a permutation of $[n]=\{1,\ldots,n\}$, Robinson--Schensted
insertion produces a pair of tableaux of the same shape, filled with
$1,\ldots,n$ in increasing order along rows and columns
\cite{Schensted}. The insertion and recording tableaux index the right and left cells,
respectively, of the symmetric group \cite{KL,Ariki}. For $a<b<c$, Knuth's interchanges
$acb\leftrightarrow cab$ and $bac\leftrightarrow bca$ of consecutive
letters preserve the insertion tableau \cite{Knuth}.
Greene recovered the row lengths from unions of increasing
subsequences \cite{Greene}. For an involution, the two tableaux
coincide and the number of fixed points equals the number of columns
of odd length \cite[Theorem~4.1]{Beissinger}.

The affine symmetric group $\tS_n$ consists of the bijections $w$
of $\ZZ$ satisfying $w(i+n)=w(i)+n$ and
$\sum_{i=1}^n w(i)=n(n+1)/2$. Shi described its left cells by an insertion
algorithm on periodic permutations \cite{ShiRS}. The affine
matrix-ball construction (AMBC) of Chmutov, Pylyavskyy and Yudovina
extends this description to a bijection \cite{CPY}. It assigns to
each affine permutation two \emph{tabloids} $P,Q$ and a weight $\rho$,
an integer vector with one coordinate per row. A tabloid distributes the
residues modulo $n$ into ordered rows of prescribed lengths, with no
order on entries within a row. The coordinates of $\rho$ distinguish
permutations having the same pair of tabloids.
Chmutov, Lewis and Pylyavskyy determined how the local Knuth moves
act on all three coordinates and how they connect tabloids of fixed
shape \cite{CLP}. Kim and Pylyavskyy used tabloids with two rows
to construct finite $W$-graphs for the affine symmetric group
\cite{KPTwoRows}.

Restricting the vertices to involutions leads to other Hecke
modules. Lusztig and Vogan constructed a module on the involutions
of a Weyl group \cite{LusztigVogan}. Rains and Vazirani constructed
two modules for certain group actions equipped with a height
function \cite{RV}. For conjugacy classes of involutions satisfying
their height axioms, Marberg constructed canonical bases and the
associated $W$-graphs \cite{Marberg}.
For even $n$, the \emph{fixed-point-free} (FPF) involutions of $\tS_n$ are the
elements $x$ with $x^2=1$ and $x(i)\ne i$ for every $i$. They
form two conjugacy classes to which Marberg's construction applies
\cite{Marberg2,Zhang}. Write $\Gamma_n^\m$ for the $W$-graph whose edge weights are
defined from the canonical-basis coefficients of the $\m$-type module
(Section~\ref{pre:fpf-graphs}).
For $\Gamma_n^\m$, the path criterion in \cite[Theorem~6.1]{ZhangFPSAC}
describes molecule membership by local transformations, and
\cite[Proposition~6.4]{ZhangFPSAC} shows that the number of molecules
at a fixed shape is independent of the occurring weight.
We give a criterion for membership directly in terms of the AMBC data.

For an involution the two affine tabloids coincide. Write $\lambda$
for their common shape, a partition of $n$ with row lengths
$\lambda_1,\dots,\lambda_\ell$, and write $d_\lambda$ for the
greatest common divisor of its column lengths. We recall the charge
statistic used to label the molecules \cite{CLP}. In adjacent rows
of the same length $s$, write the representatives in $[n]$ as
$u_1<\cdots<u_s$ in row $i$ and $v_1<\cdots<v_s$ in row $i+1$.
Let $c_i$ be the least $h\in\{0,\ldots,s\}$ such that
$u_j<v_{j+h}$ for $1\le j\le s-h$, and set $c_i=0$ when the
lengths differ. The \emph{charge} of the tabloid is $\sum_{i=1}^{\ell-1} i c_i$.

\newpage
\begin{theorem}[Classification of molecules]\label{intro-molecule-thm}
Let $n\ge4$ be even.
\begin{enumerate}
\item The shapes occurring in $\Gamma_n^\m$ are precisely the
partitions of $n$ with even column lengths. The weights at such a shape are
the integer vectors whose coordinates, on each maximal block of
equal-length rows, are weakly increasing and change sign under reversal.
\item At fixed shape $\lambda$ and weight $\rho$, two vertices
belong to the same molecule if and only if the charges of their
tabloids agree modulo $d_\lambda$. There are $d_\lambda$ molecules,
each with $n!/(d_\lambda\prod_{i=1}^{\ell}\lambda_i!)$ vertices.
For any two molecules with the same shape and arbitrary occurring
weights, their bidirected graphs are isomorphic up to a cyclic shift
of the simple-generator labels. Each directed weight on a
bidirected edge is one.
\end{enumerate}
\end{theorem}

Section~\ref{prelim-sect} defines the affine fixed-point-free
$W$-graphs and recalls the AMBC correspondence and Knuth moves on tabloids.
Section~\ref{Knuth-sect} describes the local transformations of
involutions in window notation. Section~\ref{shift-sect} uses
the fixed-point count to determine which AMBC shapes and weights occur.
It then describes how rotation acts on the involutions at each such
shape and weight.
Section~\ref{mole-sect} classifies the molecules by charge, compares
their bidirected graphs at different weights and counts their vertices.

\section*{Declaration of AI use}
During the preparation of this work, OpenAI ChatGPT and Codex were used to assist with literature searches, proof drafting and revision, the development of computational checks, and manuscript preparation.
The author is responsible for the content of the submitted manuscript.

\section{Preliminaries}\label{prelim-sect}

\subsection{Affine permutations and involutions}
Let $n$ be a positive integer. Write $\ZZ$ for the set of integers and define $[n] = \{1,2,\dots,n\}$.
Let $\NN = \{0,1,2,\dots\}$ and $\PP=\{1,2,3,\dots\}$ be the sets of nonnegative and positive integers.

\begin{definition}
The \defn{extended affine symmetric group} $\widehat S_n$ is the group
of bijections $\pi:\ZZ\to\ZZ$ satisfying
$\pi(i+n)=\pi(i)+n$ for all $i\in\ZZ$.  The \defn{affine symmetric
group} $\tS_n$ is its subgroup consisting of the elements satisfying
\[
\pi(1)+\pi(2)+\cdots+\pi(n)=1+2+\cdots+n.
\]
\end{definition}

We refer to elements of $\tS_n$ as \defn{affine permutations}.  The
\defn{standard window} of $\pi\in\widehat S_n$ is
\[
[\pi(1),\ldots,\pi(n)].
\]
Its shifted windows are $[\pi(i+1),\ldots,\pi(i+n)]$ for
$i\in\ZZ$.  Together with its starting index, any window determines $\pi$
uniquely.  A sequence
$[a_1,\ldots,a_n]$ is the standard window of an element of
$\widehat S_n$ if and only if its entries represent every congruence
class modulo $n$ exactly once.  It is the standard window of an element
of $\tS_n$ if and only if, in addition,
\[
\sum_{k=1}^n a_k=\frac{n(n+1)}2.
\]
For $\pi\in\tS_n$, the shifted window beginning at $i+1$ has sum
$n(n+1)/2+in$.

To define the Coxeter generators, assume $n\geq2$. Let $s_i$ for $i\in\ZZ$ be the unique element of $\tS_n$ that interchanges $i$ and $i+1$ while fixing every integer $j \notin \{i,i+1\} + n\ZZ$.
One has $s_i = s_{i+n}$ for all $i \in \ZZ$, and 
 $\{s_1,s_2,\dots,s_n\}$ generates the group $\tS_n$.
With respect to this generating set, $\tS_n$ is the affine Coxeter group of type $\tilde A_{n-1}$.
The parabolic subgroup $S_n = \langle s_1,s_2,\dots,s_{n-1}\rangle \subset \tS_n$
is the finite Coxeter group of type $A_{n-1}$; its elements
are
the permutations $\pi \in \tS_n$ with $\pi([n]) = [n]$.

A \defn{reduced expression} for $\pi \in \tS_n$ is a minimal-length factorization $\pi = s_{i_1}s_{i_2}\cdots s_{i_l}$.
The \defn{length} of $\pi \in \tS_n$, denoted $\ell(\pi)$, is the number of factors  in any of its reduced expressions.
Equivalently, $\ell(\pi)$ counts the equivalence classes in the set
$
\Inv(\pi) = \{ (i,j) \in \ZZ \times \ZZ : i< j\text{ and }\pi(i)> \pi(j)\}
$
under the relation $\sim$
on $\ZZ\times \ZZ$ with $(a,b) \sim (a',b')$ if and only if $a-a' = b-b' \in n \ZZ$.

We write $\widebar i$ for the residue class of $i$ modulo $n$.
A reflection is a conjugate of a simple generator. In the ordinary
Bruhat order, $x\le y$ if and only if a reduced expression for $y$
contains a reduced expression for $x$ as a subsequence
\cite[Theorem~2.2.2]{BB}.

For $w\in\tS_n$, define its left and right descent sets by
\[
L(w):=\{\widebar i\in\ZZ/n\ZZ:\ell(s_iw)<\ell(w)\},
\qquad
R(w):=\{\widebar i\in\ZZ/n\ZZ:\ell(ws_i)<\ell(w)\}.
\]

In window notation, $\widebar i\in R(w)$ exactly when
$w(i)>w(i+1)$, and $L(w)=R(w^{-1})$.

From here, assume that $n$ is even. An affine involution is an element $z\in\tS_n$ with $z^2=1$. An affine involution $z$ is fixed-point-free if there is no $x\in[n]$ with $z(x)=x$. We write $\tilde\cI_n$ and $\cF_n$ for the sets of affine involutions and affine fixed-point-free involutions, respectively. On $\cF_n$, we define $\ellfpf(z) = \tfrac{1}{2}(\ell(z) - \tfrac{n}{2})$.

For a Coxeter group $(W,S)$ acting on a set $X$, an integer-valued
height $h$ makes $X$ a \emph{scaled} $W$-set if
$|h(sx)-h(x)|\le1$ for $s\in S$. The scaled set is \emph{quasiparabolic} if,
for every reflection $t$ and every $s\in S$, the equality
$h(tx)=h(x)$ forces $tx=x$, and the two inequalities
$h(tx)>h(x)$ and $h(stx)<h(sx)$ together force $tx=sx$
\cite[Definition~2.2]{Marberg}.

The conjugation action of $\tS_n$ on $\cF_n$, together with the height
function $\ellfpf$, is a quasiparabolic set
\cite[Theorem~4.8]{Zhang}.  Its \defn{FPF Bruhat order} $\leq_F$ is
the transitive closure of the relations
\[
z\leq_F tzt
\qquad\text{whenever $t$ is a reflection of $\tS_n$ and}
\qquad
\ellfpf(z)\leq\ellfpf(tzt);
\]
see \cite[Definition~4.7]{Zhang}.  We write $<_F$ for the associated
strict order. Can, Cherniavsky and Twelbeck studied the restriction
of Bruhat order to fixed-point-free involutions in the finite symmetric
group \cite{CanFPF}.

\begin{definition}
Given $\pi \in\tS_n$, define $\beta(\pi)=\frac{1}{2n}\sum_{i=1}^n|\pi(i)-r_n(\pi(i))|,$ where $r_n(i)$ for $i \in \ZZ$ denotes the unique element of $[n]$ that satisfies $r_n(i)\equiv i\pmod n$. For $z \in \cF_n$, define $\sfpf(z)=(-1)^{\beta(z)}$.
\end{definition}
Let $\tp =s_1s_3\cdots s_{n-1}= [2,1,4,3,\dots,n,n-1] \in \tilde\cI_n$ and
$\Theta^-=s_2s_4\cdots s_n=[0,3,2,5,\ldots,n-2,n+1]\in\tilde\cI_n$,
so that $\sfpf(\Theta^\pm)=\pm 1$.
We use $\Theta$ to denote either element of $\{ \Theta^\pm\}$.
Define $\cFp$ as the $\tS_n$-conjugacy class of $\Theta^+$ and $\cFm$ as the $\tS_n$-conjugacy class of $\Theta^-$. These classes satisfy
\be\label{+-eq} \cFp =\{z\in\cF_n : \sfpf(z)=1\} \qquand
 \cFm =\{z\in\cF_n : \sfpf(z)=-1\}\ee
and hence $\cF_n=\cFp\sqcup\cFm$; see, e.g., \cite[Theorem 5.4]{Marberg2}.

The elements $\Theta^+$ and $\Theta^-$ have height zero in their
respective conjugacy classes.
For $x\in\cF_n$ and $s=s_i$, the height change
$\ellfpf(sxs)-\ellfpf(x)$ equals $1$ if $x(i)<x(i+1)$,
equals $0$ if $x(i)=i+1$, and equals $-1$ in the remaining case
\cite[Corollary~3.16]{Zhang}.

\subsection{Affine \texorpdfstring{$\fpf$ $W$}{FPF W}-graphs and molecules}\label{pre:fpf-graphs}
The graph $\Gamma_n^\m$ has one vertex for each fixed-point-free
involution in $\tS_n$. The descent labels of two vertices and the
coefficients of the canonical basis in the standard basis determine
whether an edge joins them in both directions.

We use Stembridge's conventions for labeled graphs \cite{Stem1,Stem2}.
Set $\cA=\ZZ[v,v^{-1}]$.  For a Coxeter system $(W,S)$ with length
function $\ell$, let $\cH=\cH(W,S)$ be its Iwahori--Hecke algebra over
$\cA$.  Thus $\cH$ has the standard basis $\{H_w:w\in W\}$ determined
by
\[
H_sH_w=
\begin{cases}
H_{sw},&\ell(sw)>\ell(w),\\
H_{sw}+(v-v^{-1})H_w,&\ell(sw)<\ell(w),
\end{cases}
\qquad(s\in S,w\in W).
\]

\begin{definition}
An \defn{$I$-labeled graph} for a finite set $I$ is a triple $\Gamma=(V,\omega,\nu)$ where
\begin{itemize}
\item[(i)] $V$ is a vertex set;
\item[(ii)] $\omega: V\times V\to \cA$ is a locally finite map, meaning
that for each $x\in V$ only finitely many $y\in V$ satisfy
$\omega(x,y)\neq0$;
\item[(iii)] $\nu: V\to \mathcal{P}(I)$ is a map assigning a subset of $I$ to each vertex.
\end{itemize}
We view $\Gamma$ as a weighted directed graph on the vertex set $V$ with an edge  $x\xrightarrow{\omega(x,y)} y$ if $\omega(x,y)\neq 0$.  An edge between $x$ and $y$ is \defn{bidirected} if both $\omega(x,y)$ and $\omega(y,x)$ are nonzero.
\end{definition}

\begin{definition}
An $S$-labeled graph $\Gamma=(V,\omega,\nu)$ is a \defn{$W$-graph} if the free $\cA$-module generated by $V$ can be given an $\cH$-module structure with
\[
H_sx=\begin{cases}vx&s\not\in\nu(x)\\
-v^{-1}x+\sum_{y\in V;s\not\in\nu(y)}\omega(x,y)y&s\in\nu(x)\end{cases}
\qquad \text{ for }s\in S\text{ and }x\in V.
\]
\end{definition}

\begin{example}
The \defn{bar operator} $H\mapsto\overline H$ of $\cH$ is the unique ring homomorphism sending $v$ to $v^{-1}$ and each $H_s$ to $H_s^{-1}$.
Write $<$ for the Bruhat order on $W$.
Kazhdan and Lusztig proved \cite{KL} that for each $w\in W$
there is a unique $\underline H_w\in\mathcal{H}$ with
$$\underline{H}_w = \overline{\underline{H}}_w\in H_w+\sum_{y<w}v^{-1}\ZZ[v^{-1}]H_y.$$
The elements $\{\underline{H}_w\}_{w\in W}$ form an $\cA$-basis for $\mathcal{H}$, called the \defn{Kazhdan-Lusztig basis}.
Define the coefficients $h_{x,y}\in\ZZ[v^{-1}]$ for $x,y \in W$ by $\underline{H}_y=\sum_{x\in W}h_{x,y}H_x$,
%Then $h_{x,y}=0$ if $x\nless y$ and $h_{1,s}=1$ for $s\in S$. Define 
and let $\mu(x,y)$ be the coefficient of $v^{-1}$ in $h_{x,y}$. 
Finally, let 
$$ \nu(x) = \{ s \in S : \ell(sx) < \ell(x)\}\qquand \omega(x,y)=\begin{cases}\mu(x,y)+\mu(y,x)&\text{if }\nu(x)\not\subseteq\nu(y)\\0,&\text{otherwise}\end{cases}$$
for $x,y \in W$.
Then $\Gamma = (W,\omega,\nu)$ is a $W$-graph \cite{KL}.
\end{example}

We now specialize to $W=\tS_n$. Let $\cM =\cA\text{-span}\{M_z:z\in\cF_n\}$ and $
\cN=\cA\text{-span}\{N_z:z\in\cF_n\}$
denote the free $\cA$-modules with bases given by the symbols $M_z$ and $N_z$ for $z\in\cF_n$. We call $\{M_z\}_{z\in\cF_n}$ and $\{N_z\}_{z\in\cF_n}$ the standard bases of $\cM$ and $\cN$, respectively. 
Let $S = \{s_1,s_2,\dots,s_n\}$.

For $k\in\ZZ$, let $\alpha_k:S\to S$ be the cyclic Dynkin-diagram
automorphism
\[
\alpha_k(s_i)=s_{i+k}.
\]
If $\Gamma=(V,\omega,\nu)$ and
$\Gamma'=(V',\omega',\nu')$ are $S$-labeled weighted directed graphs,
an \defn{$\alpha_k$-twisted isomorphism} is a bijection
$f:V\to V'$ such that
\[
\nu'(f(x))=\alpha_k(\nu(x))
\qquad\text{and}\qquad
\omega'(f(x),f(y))=\omega(x,y)
\]
for all $x,y\in V$.  We use the same terminology for the graphs
obtained by retaining only bidirected edges and their two directed
weights.

\begin{proposition}[{\cite[Corollary~4.14 in the preprint version]{Zhang}}]\label{fpf-module-action}
The modules $\cM$ and $\cN$ have unique $\cH$-module structures satisfying, for $s\in  S$ and $z\in\cF_n$, the formulas 
{\small\[H_sM_z=\begin{cases}
M_{szs}&\ellfpf(szs)>\ellfpf(z)\\
vM_{z}&\ellfpf(szs)=\ellfpf(z)\\
M_{szs}+(v-v^{-1})M_z&\ellfpf(szs)<\ellfpf(z)\end{cases}
\]}
and
{\small\[H_sN_z=\begin{cases}
N_{szs}&\ellfpf(szs)>\ellfpf(z)\\
-v^{-1}N_{z}&\ellfpf(szs)=\ellfpf(z)\\
N_{szs}+(v-v^{-1})N_z&\ellfpf(szs)<\ellfpf(z).\end{cases}
\]}
\end{proposition}

These actions arise from the quasiparabolic modules of Rains and Vazirani \cite{RV}. The bar operators and canonical bases below come from \cite{Marberg}.

We write $f \mapsto \overline{f}$ for the automorphism of $\cA$ interchanging $v$ and $v^{-1}$. An additive map $F:U\to V$ between $\cA$-modules is \defn{$\cA$-antilinear} if $F(au)=\overline{a}F(u)$ for all $a\in\cA$ and $u\in U$.

\begin{corollary}[{\cite[Corollary~4.15 in the preprint version]{Zhang}}]\label{fpf-bar-canonical}
The $\cH$-modules $\cM$ and $\cN$ have the following properties:
\begin{itemize}
\item[(a)] There are unique $\cA$-antilinear maps $\cM\to\cM$ and $\cN\to\cN$, denoted by $X \mapsto\overline X$,
such that 
$$\overline{HM}=\overline{H}\cdot\overline{M}\text{ and }\overline{M_{\Theta}}=M_{\Theta}
\qquand
\overline{HN}=\overline{H}\cdot\overline{N}\text{ and }\overline{N_{\Theta}}=N_{\Theta}
$$
for all $M \in \cM$, $N \in \cN$, and $\Theta \in \{\Theta^\pm\}$. Moreover, both of these maps are involutions.

\item[(b)] The $\mathcal{H}$-modules
$\cM$ and $\cN$ have unique $\cA$-bases $\{\underline{M}_x\}_{x \in \cF_n}$ and $\{\underline{N}_x\}_{x \in \cF_n}$ satisfying
$$\underline{M}_x=\overline{\underline{M}_x}\in M_x+\sum_{w<_F x}v^{-1}\ZZ[v^{-1}]M_w
\qquand
\underline{N}_x=\overline{\underline{N}_x}\in N_x+\sum_{w<_F x}v^{-1}\ZZ[v^{-1}]N_w$$
where both sums are over $w\in\cF_n$. These are the \defn{canonical bases} of $\cM$ and $\cN$.

\end{itemize}
\end{corollary}

Define $\m_{x,y}$ and $\n_{x,y}$ for $x,y\in \cF_n$ as the polynomials in $\ZZ[v^{-1}]$ such that
$$\underline{M}_y=\sum_{x\in\cF_n}\m_{x,y}M_x\qquand
\underline{N}_y=\sum_{x\in\cF_n}\n_{x,y}N_x.$$
Let $\mu_\m(x,y)$ and $\mu_\n(x,y)$ denote the coefficients of $v^{-1}$ in $\m_{x,y}$ and $\n_{x,y}$.
Define $\nu_\m$, $\nu_\n$: $\cF_n\to \mathcal{P}(S)$ by
\[
\nu_\m(x) = \{s\in S : sxs \le_F x\}\qquand \nu_\n(x) = \{s\in S : x\le_F sxs\}
\]
where $S = \{s_1,s_2,\dots,s_n \} \subset \tS_n$.
Finally, let $\omega_\m$: $\cF_n\times\cF_n\to\ZZ$ be the map with
\[
\omega_\m(x,y)=\begin{cases}\mu_\m(x,y)+\mu_\m(y,x)&\nu_\m(x)\not\subseteq\nu_\m(y)\\
0&\nu_\m(x)\subseteq\nu_\m(y).\end{cases}
\]
Define $\omega_\n$: $\cF_n\times\cF_n\to\ZZ$ by the same formula, but with $\mu_\m$ and $\nu_\m$ replaced by $\mu_\n$ and $\nu_\n$.

Write $[E]=1$ if $E$ holds and $[E]=0$ otherwise.
For $x<_F y$ let $s$ be a simple reflection. If
$sys\le_F y$ and $sxs>_F x$, then $\mu_\m(x,y)=[sxs=y]$.
If $sys<_F y$ and $sxs\ge_F x$, then $\mu_\n(x,y)=[sxs=y]$.
These coefficient rules are \cite[Lemma~3.25]{Marberg}.

\begin{corollary}[{\cite[Theorem 3.26]{Marberg}}]
Both $\Gamma^\m_n=(\cF_n,\omega_\m,\nu_\m)$ and $\Gamma^\n_n=(\cF_n,\omega_\n,\nu_\n)$ are $\tS_n$-graphs.
\end{corollary}

We call these graphs \defn{affine $\fpf$ graphs}. The strongly connected components in a $W$-graph $\Gamma$ are called \defn{cells}. The connected components with respect to bidirected edges are called \defn{molecules}.

\subsection{Tabloids}
Set $[\ol n]=\{\ol1,\ldots,\ol n\}$. A partition
$\lambda=(\lambda_1,\ldots,\lambda_\ell)$ is a weakly decreasing
sequence of positive integers. We write $\ell(\lambda)=\ell$ for
its number of rows, $|\lambda|=\sum_i\lambda_i$ for its size and
$\lambda\vdash n$ for $|\lambda|=n$. Its Young diagram has
$\lambda_i$ boxes in row $i$, counted from the top; the conjugate
partition has parts $\lambda'_j=\#\{i:\lambda_i\ge j\}$.
For $|\lambda|\le n$, a \defn{tabloid} $P$ of shape $\lambda$ is an ordered list
of pairwise disjoint sets $P_1,\ldots,P_\ell\subseteq[\ol n]$ with
$|P_i|=\lambda_i$.  We draw $P_i$ as the $i$-th row of the Young
diagram.  Reordering entries within a row gives the same tabloid.
When $|\lambda|=n$, every residue occurs exactly once.

\begin{example}
\[
\yttab{\ol{1}&\ol{2}&\ol{3} \\ \ol{4}} \qquad 
\yttab{\ol{3}&\ol{1}&\ol{2} \\ \ol{4}} \qquad 
\yttab{\ol{4}&\ol{2}&\ol{3}\\ \ol{1}} \qquad 
\yttab{\ol{4}&\ol{1}&\ol{2} \\\ol{3}}\qquad 
\yttab{\ol{1}&\ol{4}&\ol{3}\\ \ol{2}} \qquad 
\]
These are several tabloids of shape $(3,1)$.  The first two tabloids are equal, since they differ only by permuting elements within rows.
\end{example}

We display each row in the order $\ol1<\cdots<\ol n$, giving
each tabloid a unique \defn{row-standard} representative.  For
$\lambda\vdash n$, write $\RSYT(\lambda)$ for the set of tabloids of
shape $\lambda$ in this presentation.

We call the tabloid analogue of a descent set the \defn{$\tau$-invariant}, following the terminology of Vogan's (generalized) $\tau$-invariant \cite{Vogan}.
\begin{definition}
For a tabloid $T$ filled with all the elements of $[\ol{n}]$, define the \defn{$\tau$-invariant} by
\[
\tau(T) := \{\ol{i}\in [\ol{n}]: \ol{i} \text{ lies in a strictly higher row of } T \text{ than } \ol{i + 1}\}.
\]
\end{definition}
\begin{example}
\label{ex:tabloid tau}
The following tabloid  $T$ has $\tau(T) = \{\ol{2},\ol{5}, \ol{8}\}$.
\[
T := \yttab{\ol{2}&\ol{4}&\ol{5}\\ \ol{6}&\ol{7}&\ol{8}\\ \ol{1}&\ol{3}\\}
\]
\end{example}

For $\lambda\vdash n$ and $a\in\mathbb Z/n\mathbb Z$, the \defn{column-superstandard
tabloid} $C_a$ of shape $\lambda$ is obtained by filling the Young
diagram column by column from left to right, and within each column
from top to bottom, with the successive residues
$\widebar a,\overline{a+1},\ldots,\overline{a+n-1}$, and then
forgetting the order within each row.  This is the convention of
\cite[Section~2]{CLP}.

\subsection{AMBC}
The AMBC conventions below apply to every $n\geq2$, with $n$
assumed even when we specialize to $\cF_n$.
We follow the local-charge and offset notation of
\cite[Section~2.2]{KP}.  Distinct tabloids $P$ and $Q$ arise in
the intermediate Knuth moves used later, so we allow them here.

Let $T\in\RSYT(\lambda)$, and suppose that
$\lambda_i=\lambda_{i+1}=m$.  Write the representatives of the two rows
in $[n]$ as
\[
T_i=(a_1<\cdots<a_m),
\qquad
T_{i+1}=(b_1<\cdots<b_m).
\]
The \defn{local charge} $\lch_i(T)$ is the smallest
$d\in\{0,\ldots,m\}$ such that
\[
a_{j-d}<b_j\qquad(d+1\leq j\leq m).
\]
Thus $d$ is the smallest rightward shift of the upper row that makes
every overlapping column strictly increasing.  This is equivalent to
the charge-matching definition in
\cite[Definitions~5.2 and~5.3]{CLP}.

For $P,Q\in\RSYT(\lambda)$, define their \defn{symmetrized offset
vector} $\offset(P,Q)=(s_1,\ldots,s_{\ell(\lambda)})$ by
\[
s_i=
\begin{cases}
0,&i=1\text{ or }\lambda_{i-1}>\lambda_i,\\
s_{i-1}+\lch_{i-1}(P)-\lch_{i-1}(Q),
   &\lambda_{i-1}=\lambda_i.
\end{cases}
\]
The successive differences agree with the offset constants of
\cite[Definition~5.8]{CPY}, as shown in \cite[Theorem~5.10]{CLP}.
The definition gives
\[
\offset(P,Q)+\offset(Q,R)=\offset(P,R),
\qquad
\offset(P,P)=0.
\]

\begin{example}\label{ex:1}
Take $n=4$, $\lambda=(2,2)$, and
\[
P=\yttab{\ol2&\ol4\\\ol1&\ol3},
\qquad
Q=\yttab{\ol1&\ol2\\\ol3&\ol4}.
\]
For $P$, the overlap at shift $d=0$ contains the decreasing pair
$2>1$, so the smallest allowed shift is $d=1$, where the only
overlapping column satisfies $2<3$.
The shift is zero for $Q$, whose columns satisfy $1<3$ and $2<4$.
Consequently
\[
\lch_1(P)=1,\qquad \lch_1(Q)=0,
\qquad \offset(P,Q)=(0,1).
\]
The two-row example in Section~\ref{size-sect} uses the same tabloid $P$.
\end{example}

A vector $\rho\in\ZZ^{\ell(\lambda)}$ is \defn{dominant with respect
to $(P,Q)$} if $\rho-\offset(P,Q)$ is weakly increasing on each
maximal block of equal parts of $\lambda$.  Equivalently,
\[
\rho_{i+1}-\rho_i
\geq \lch_i(P)-\lch_i(Q)
\qquad\text{whenever }\lambda_i=\lambda_{i+1}.
\]
We write $\rev_\lambda(\rho)$ for the vector obtained by reversing
the coordinates within each such block.  Thus dominance is equivalent
to requiring $\rev_\lambda(\rho-\offset(P,Q))$ to be weakly
decreasing on each block, as in \cite[Section~2.2]{KP}.
For an arbitrary $\rho$, its \defn{dominant representative} $\rho'$
with respect to $(P,Q)$ is obtained by sorting
$\rho-\offset(P,Q)$ increasingly within each equal-part block and
then adding $\offset(P,Q)$ back; see
\cite[Definition~6.1]{CPY} and \cite[Section~5.3]{CLP}.

\begin{example}\label{ex:2}
For the pair in Example~\ref{ex:1}, dominance means
$\rho_2-\rho_1\geq1$.  The zero-sum vector $\rho=(-1,1)$ is dominant,
since
\[
\rho-\offset(P,Q)=(-1,0).
\]
The vector $(0,0)$ is not dominant: subtracting the offset gives
$(0,-1)$.  Sorting this to $(-1,0)$ and adding $(0,1)$ back gives
its dominant representative $(-1,1)$.
\end{example}

Set
\begin{align*}
\Omega&\colonequals
\bigsqcup_{\lambda\vdash n}
\RSYT(\lambda)\times\RSYT(\lambda)\times\ZZ^{\ell(\lambda)},\\
\widehat{\dom}&\colonequals
\{(P,Q,\rho)\in\Omega:\rho\text{ is dominant with respect to }(P,Q)\},\\
\dom&\colonequals
\{(P,Q,\rho)\in\widehat{\dom}:\textstyle\sum_i\rho_i=0\}.
\end{align*}
The affine matrix-ball construction gives a bijection
\[
\Phi:\widehat S_n\longrightarrow\widehat{\dom},
\qquad
w\longmapsto(P(w),Q(w),\rho(w)).
\]
Backward AMBC defines a surjection
$\Psi:\Omega\to\widehat S_n$ with
$\Psi|_{\widehat{\dom}}=\Phi^{-1}$; see
\cite[Theorems~5.1, 5.11, and~5.12]{CPY}.
More generally, \cite[Theorem~6.3]{CPY} gives
\[
\Phi(\Psi(P,Q,\rho))=(P,Q,\rho').
\]
The restriction of $\Phi$ is a bijection $\tS_n\to\dom$, because
\cite[Lemma~4.6]{CLP} gives
\[
\sum_i\rho_i(w)=\frac1n\sum_{j=1}^n(w(j)-j).
\]
In particular, the dominant triple in Example~\ref{ex:2} belongs to
$\dom$ and hence corresponds to an element of $\tS_4$.

The following properties describe inversion and descent sets.
\begin{proposition}[{\cite[Proposition~3.1]{CLP}}]
\label{ambc-inverse-prop}
If $\Phi(w)=(P,Q,\rho)$, then
\[
\Phi(w^{-1})=(Q,P,(-\rho)'),
\]
where the dominant representative on the right is taken with respect
to $(Q,P)$, using the offset $\offset(Q,P)$.
In particular, $w=w^{-1}$ implies $P(w)=Q(w)$.
\end{proposition}
The last assertion follows by comparing the first two coordinates.
For $x\in\cF_n$, we therefore have $\Phi(x)=(P,P,\rho)$;
the offset is zero, and $\rho$ is weakly increasing within each
equal-part block.

\begin{proposition}[{\cite[Proposition~3.6]{CLP}}]
\label{prop:descents}
For $w\in\tS_n$, one has
$L(w)=\tau(P(w))$ and $R(w)=\tau(Q(w))$.
\end{proposition}

\subsection{Insertion shapes and streams}\label{pre:streams}

For a finite word of distinct integers, \emph{row insertion} places
the next letter in the first row, bumping the leftmost larger entry
into the next row and continuing until an entry is appended. The
row lengths of the resulting tableau give its insertion shape.
The recording tableau labels each newly created box by the
insertion step at which it appears.
Replacing each letter by its rank among the letters of the word
is called standardization and preserves the shape.
For insertion shape $\eta$, $\eta_1+\cdots+\eta_r$ is the largest
number of positions covered by $r$ disjoint increasing subsequences,
and $\ell(\eta)$ is the length of a longest decreasing subsequence
\cite{Greene}. For a finite involution the number of fixed points
is the number of odd columns of $\eta$
\cite[Theorem~4.1]{Beissinger}.

On $S_n$, AMBC gives the
row sets of the ordinary insertion and recording tableaux, with
weight zero \cite[Theorem~10.2]{CPY}. More generally, if the
AMBC shape of $w$ is $\lambda$ and $\eta^{(k)}$ is the insertion
shape of $w(1),\ldots,w(kn)$, then for all sufficiently large $k$,
$\eta_i^{(k+1)}-\eta_i^{(k)}=\lambda_i$ for every row $i$,
with absent parts set to zero \cite[Theorem~7.3]{CPY}.

We also allow partial permutations: the domain may be any union of
residue classes, and each integer has at most one preimage.
Draw such a periodic injection as its set of
\emph{balls} $(i,w(i))$, with row indices increasing downwards
and column indices to the right. A \emph{stream} is a nonempty set
of points invariant under translation by $(n,n)$ in which both
coordinates strictly increase together. Its density is the number
of points modulo this translation. Enumerate its row and column
coordinates as increasing sequences $a_j,b_j$, with the first
$s$ terms of each in $[n]$ for a stream of density $s$. Thus
$a_{j+s}=a_j+n$ and $b_{j+s}=b_j+n$. The stream consists of
$(a_j,b_{j+h})$ for one integer $h$, called its \emph{altitude}.

A \emph{channel} is a stream of balls of maximum density. Channel
$C$ is weakly southwest of channel $C'$ if each ball in $C$ has a
ball of $C'$ weakly northeast of it. The least channel for this
order is the southwest channel. Number its balls consecutively
in increasing row order. The channel number of a ball $b$ is the
maximum of $j+k$ over all strictly southeast paths of $k$ steps
from a channel ball numbered $j$ to $b$; both coordinates increase
strictly at each step of such a path. Adding a constant to the
initial numbering translates all the labels by that constant
\cite[Definitions~3.8 and~3.12--3.14]{CPY}.

For a given label, list its balls as $(p_1,c_1),\ldots,(p_t,c_t)$
with $p_1<\cdots<p_t$; then $c_1>\cdots>c_t$. One forward AMBC
step replaces these balls by $(p_1,c_t)$ in the output stream and
$(p_{j+1},c_j)$ for $1\le j<t$ in the remaining partial permutation. The row and column residues of
the output stream give the next rows of $Q$ and $P$, respectively,
and its altitude gives the next entry of $\rho$. AMBC repeats this
step on the remaining balls \cite[Definition~3.26]{CPY}.

\subsection{Knuth moves, rotation and charge}\label{pre:knuth}

We assume $n\ge4$ for the rank-two moves and the classification of
molecules, since at $n=2$ the product of the two simple generators
has infinite order. All subscripts of simple
reflections and tabloid operators are read modulo $n$.  In inequalities
such as ``$w(i-1)$ is between $w(i)$ and $w(i+1)$'', we use the
consecutive integer representatives $i-1,i,i+1$; replacing $i$
by $i+kn$ leaves the condition unchanged.

\begin{definition}
For $w\in\tS_n$ and $i\in\ZZ$, the \defn{right Knuth move} at
position $\widebar i$ interchanges the adjacent entries $w(i)$ and
$w(i+1)$ and all their periodic translates if $w(i-1)$ or $w(i+2)$
lies strictly between the two entries. In this case the result is
$w'=ws_i$. Otherwise the move is trivial and $w'=w$. We write
$w\K{\widebar i}w'$ in both cases \cite[Section~3]{CLP}.
\end{definition}

\begin{definition}
For a tabloid $T$ and $i\in\ZZ$, let $s_i(T)$ be obtained by
interchanging the residues $\widebar i$ and $\overline{i+1}$.  The
\defn{tabloid Knuth operator} $D_i$ is
\[
D_i(T):=\begin{cases}
s_i(T),&\text{if $\tau(T)$ and $\tau(s_i(T))$ are incomparable},\\
T,&\text{otherwise.}
\end{cases}
\]
Thus $D_{i+n}=D_i$.  A nontrivial pair $T,D_i(T)$ is an edge of
the tabloid Knuth graph of \cite[Definition~3.9]{CLP}.
\end{definition}

For a partition $\lambda\vdash n$, let $\mathcal A_\lambda$ be the
simple graph with vertex set $\RSYT(\lambda)$ in which $T$ is adjacent
to $D_i(T)$ whenever $D_i(T)\neq T$.  We write
$T\sim_D T'$ when $T$ and $T'$ lie in the same connected component of
$\mathcal A_\lambda$.

\begin{definition}
Let $\bar a$ belong to row $t$ of a tabloid $T$ of shape $\lambda$.
Following \cite[Definition~5.2]{KP}, define two vectors of length
$\ell(\lambda)$. The vector $\iota(T,a)$ is $1$ in coordinate $t$
and $0$ elsewhere. If row $t$ is the first in its block of equal
lengths, set $\delta(T,a)_j=[\lambda_j=\lambda_t]$ for every $j$.
For any other row in that block set $\delta(T,a)=0$.
\end{definition}

\begin{theorem}[{\cite[Theorem 3.11]{CLP}}]\label{Kthm}
Suppose that $w'$ is obtained from $w$ by a nontrivial Knuth move
and $\Phi(w)=(P,Q,\rho)$.  Then $Q(w')=D_k(Q)\ne Q$ for some $k$.
With $k$ indexing the residues interchanged in the recording tabloid,
\[
 \Phi(w')=
 \begin{cases}
 (P,D_k(Q),\rho),&\widebar k\ne\widebar n,\\
 (P,D_n(Q),\rho+\iota(Q,1)-\iota(Q,n)),&\widebar k=\widebar n .
 \end{cases}
\]
\end{theorem}

\begin{definition}
Let $i\in\ZZ$. Two affine permutations $w$ and $w'$ are connected by a \defn{dual Knuth move} at position $\widebar i$ if $w^{-1}$ and $(w')^{-1}$ are connected by a Knuth move at position $\widebar i$.
\end{definition}
We write $w\dK{\widebar i}w'$ for this move.  We also write
$w\dK{\widebar i}w$ when the Knuth move at $\widebar i$ is
trivial on $w^{-1}$.

The tabloid graph also describes the right Knuth moves available
from each window. If exactly one of $\overline{i-1},\widebar i$ belongs to
$\tau(T)$, there is a unique neighbor in $\mathcal A_\lambda$ which
reverses membership of these two residues. For every $w$ with
$Q(w)=T$, this edge lifts uniquely to a right Knuth move incident
to $w$ \cite[Lemma~3.22 and Proposition~3.23]{CLP}.
For a tabloid swap $P'=D_k(P)\ne P$,
\cite[Lemma~5.9]{KP} gives $\offset(P',Q)=\offset(P,Q)$ for
$\widebar k\ne\widebar n$ and
\[
 \offset(P',Q)-\offset(P,Q)
 =-\iota(P,1)+\iota(P,n)+\delta(P,1)-\delta(P,n)
\]
for $\widebar k=\widebar n$.

For $i\in\ZZ$, set
\[
\mathcal D_R(i)=
\left\{w\in\tS_n:
|R(w)\cap\{\overline{i-1},\widebar i\}|=1\right\}
\]
and define $\mathcal D_L(i)$ in the same way using $L(w)$.
For $w\in\mathcal D_R(i)$, exactly one element of
\[
\{ws_{i-1},ws_i\}
\]
belongs to $\mathcal D_R(i)$; denote it by $\mathsf S_i^R(w)$.
Set $\mathsf S_i^R(w)=w$ when $w\notin\mathcal D_R(i)$.
Similarly, for $w\in\mathcal D_L(i)$, let $\mathsf S_i^L(w)$ be
the unique element of
\[
\{s_{i-1}w,s_iw\}\cap\mathcal D_L(i),
\]
and set $\mathsf S_i^L(w)=w$ outside $\mathcal D_L(i)$.  These
are the right and left rank-two star operations for the adjacent
pair $\{s_{i-1},s_i\}$.

Let $\omega\in\widehat S_n$ be the extended affine permutation with
standard window
\[
\omega=[2,3,\ldots,n,n+1].
\]
Thus $\omega(i)=i+1$ for all $i\in\ZZ$.  Although
$\omega\notin\tS_n$, it normalizes $\tS_n$ and satisfies
$\omega s_i\omega^{-1}=s_{i+1}$, with subscripts read modulo $n$.
For a tabloid $T$, define $\omega(T)$ by replacing every residue
$\widebar a$ in $T$ by $\overline{a+1}$; this does not change the
shape of $T$.
The extended version of AMBC gives the following formulas.

\begin{proposition}[{\cite[Proposition 5.5]{KP}}]\label{omega_prop}
Suppose $w$ is an affine permutation with $\Phi(w)=(P,Q,\rho)$. Then 
\begin{align*}
\Phi(\omega w)&=(\omega(P),Q,\rho-\offset(P,Q)+\offset(\omega(P),Q)+\delta(P,n)),\\
\Phi(w\omega^{-1})&=(P,\omega(Q),\rho-\offset(P,Q)+\offset(P,\omega(Q))-\delta(Q,n)).
\end{align*}
Moreover, if $w$ is an involution, then $\Phi(\omega w\omega^{-1})=(\omega(P),\omega(P),\rho)$.
\end{proposition}
On residues, $\omega(P)=s_1s_2\cdots s_{n-1}P$.

For a partition $\lambda\vdash n$, write $\lambda'$ for its
conjugate partition and put
\[
d_\lambda
:=
\gcd(\lambda'_1,\lambda'_2,\ldots).
\]
Equivalently, $d_\lambda$ is the greatest common divisor of the multiplicities of the distinct parts of $\lambda$.

Extend the local charge convention by setting $\lch_i(T)=0$ whenever
$\lambda_i\neq\lambda_{i+1}$.  Following
\cite[Definition~8.3]{CLP}, define the \defn{charge} of a tabloid
$T$ of shape $\lambda$ by
\[
\operatorname{charge}(T)
:=
\sum_{i=1}^{\ell(\lambda)-1}i\,\lch_i(T).
\]

Every component of $\mathcal A_\lambda$ contains some $C_a$.
Two tabloids lie in the same component if and only if their charges
are congruent modulo $d_\lambda$, and
$\operatorname{charge}(C_{a+1})\equiv\operatorname{charge}(C_a)-1
\pmod{d_\lambda}$
\cite[Lemmas~7.4 and~8.5 and Theorem~8.6]{CLP}. Since translation
by $\omega$ sends $D_i$ to $D_{i+1}$ by conjugation and takes
$C_a$ to $C_{a+1}$, it decreases the charge of every tabloid by
one modulo $d_\lambda$.

\section{Knuth moves on affine involutions}\label{Knuth-sect}

Theorem~\ref{Kthm} gives the change in AMBC weight under a right
Knuth move. We use the inversion formula of Proposition~\ref{ambc-inverse-prop}
to obtain the corresponding formula for a left move. We then combine
the right and left stars for a pair of adjacent generators.

\begin{proposition}[{\cite[Lemma~5.9 and Proposition~5.10]{KP}}]\label{dKprop}
Let $w'$ be obtained from $w$ by a nontrivial dual Knuth move
and write $\Phi(w)=(P,Q,\rho)$.  Then the resulting insertion tabloid is
$P'=D_k(P)\ne P$ for some $k$, and $\Phi(w')=(P',Q,\rho-c)$, where
\[
 c=\begin{cases}
  0,&\widebar k\ne\widebar n,\\
  \iota(P,1)-\iota(P,n),&\widebar k=\widebar n.
 \end{cases}
\]
\end{proposition}

\begin{proof}
Put $o=\offset(P,Q)$, $o'=\offset(P',Q)$ and $\sigma=\rho-o$.
By Proposition~\ref{ambc-inverse-prop}, inversion interchanges
the tabloids and sends the normalized coordinate $\sigma$ to
$-\rev_\lambda(\sigma)$.  Apply Theorem~\ref{Kthm} to the right
move on $w^{-1}$.  This move preserves $Q$, sends $P$ to $P'$, and
changes the raw weight by $c$.  Since $\offset(Q,P')=-o'$, its normalized
weight change is $d=c+o'-o$.  The tabloid offset formula
\cite[Lemma~5.9]{KP} gives $d=0$ in the interior and
$d=\delta(P,1)-\delta(P,n)$ at the boundary.  In both cases
$d$ is constant on each equal-width group.  Inverting again
therefore gives normalized weight $\sigma-d$, and adding $o'$
gives the required weight $\rho-c$.
\end{proof}

\begin{definition}
Let $i\in\ZZ$.  Two affine involutions $x$ and $y$ are connected
by an \defn{involutive transformation of type $\widebar i$} if
\[
y=\mathsf S_i^L\bigl(\mathsf S_i^R(x)\bigr).
\]
We denote this relation by $x\simRBS{\widebar i}y$.
\end{definition}

Thus the label $\widebar i$ refers to the rank-two pair
$\{s_{i-1},s_i\}$; the Knuth move actually used by a nontrivial
transformation may occur at $\overline{i-1}$ or at $\widebar i$.

Let $t_{i-1,i+1}=s_{i-1}s_is_{i-1}$, the affine
transposition interchanging $i-1+kn$ and $i+1+kn$ for every
$k\in\ZZ$.

\begin{theorem}\label{direct-prop}
Let $x,y\in\tilde\cI_n$ satisfy $x\simRBS{\widebar i}y$, and put
$A=\{i-1,i,i+1\}$.  Then
\[
y=\begin{cases}
x,&x(A)\neq A\text{ and $x(i)$ is between $x(i-1)$ and $x(i+1)$},\\
s_{i-1}xs_{i-1},&x(A)\neq A\text{ and $x(i+1)$ is between $x(i-1)$ and $x(i)$},\\
s_ixs_i,&x(A)\neq A\text{ and $x(i-1)$ is between $x(i)$ and $x(i+1)$},\\
t_{i-1,i+1}x t_{i-1,i+1},&x(A)=A.
\end{cases}
\]
\end{theorem}

\begin{proof}
Write $p=i-1$, $q=i$, $r=i+1$, and
$(a,b,c)=(x(p),x(q),x(r))$.  The right descents at $p,q$ are
determined by the comparisons $a>b$ and $b>c$.
For $x(A)\neq A$, the multipliers in the two star operations are
as follows; the entry $1$ means that the operation is trivial.
We first apply the right multiplier to $x$, then apply the left
multiplier to the resulting permutation.
\[
\begin{array}{c|c|c|c}
\text{order of the three values}
  &R(x)\cap\{\widebar p,\widebar q\}
  &\text{right multiplier}&\text{left multiplier}\\ \hline
a<b<c&\varnothing&1&1\\
a<c<b&\{\widebar q\}&s_p&s_p\\
b<a<c&\{\widebar p\}&s_q&s_q\\
b<c<a&\{\widebar p\}&s_p&s_p\\
c<a<b&\{\widebar q\}&s_q&s_q\\
c<b<a&\{\widebar p,\widebar q\}&1&1
\end{array}
\]
We obtain the right-multiplier column by swapping the first two
or the last two values and testing membership in $\mathcal D_R(i)$.
In the first and last rows, $x\notin\mathcal D_R(i)$; since
$x=x^{-1}$, the left operation is also trivial.

To verify the left-multiplier column, we account for translates
across the boundary of a window.  Suppose first that $c$ is between
$a$ and $b$, so that the right operation gives $z=xs_p$.
Since $z^{-1}=s_px$, its left star operation is determined by the
relative order of
\[
s_p(a),\quad s_p(b),\quad s_p(c).
\]
As a permutation of $\ZZ$, $s_p$ preserves the order of any two
integers unless they form one of the pairs
$\{p+kn,q+kn\}$, $k\in\ZZ$.  Since $c$ is between $a$ and $b$,
the pair $(a,b)$ cannot be such a pair.  Thus the relative order
can change only if $(a,c)$ or $(b,c)$ is one of these pairs.
There are four oriented possibilities:
\[
\begin{array}{c|l}
\text{possible values}&\text{consequence of $x^2=1$ and periodicity}\\ \hline
a=p+kn,\ c=q+kn&k=0,\quad (a,b,c)=(p,r,q)\\
a=q+kn,\ c=p+kn&2kn=r-q=1\\
b=p+kn,\ c=q+kn&2kn=r-p=2\\
b=q+kn,\ c=p+kn&k=0,\quad (a,b,c)=(r,q,p)
\end{array}
\]
For example, in the second row,
$x(c)=x(p+kn)=a+kn=r$, which gives $q+2kn=r$.
In the first and fourth rows we also use the fact that
$x(u)=u+kn$ forces $k=0$, by applying $x$ a second time.
The middle two rows are impossible because $n\geq4$.
The last row contradicts the assumption that $c$ is the middle
value.  The first row gives $x(A)=A$, which is excluded here.
Consequently $s_p$ preserves the relative order of $a,b,c$.
The same right-multiplier test, applied to $z^{-1}=s_px$, selects
$s_p$; hence the left operation sends $z$ to $s_pz=s_pxs_p$.

If $a$ is between $b$ and $c$, the right multiplier is $s_q$.
The corresponding check uses the pairs $\{q+kn,r+kn\}$.
The four possible order-changing configurations give
\[
\begin{array}{c|l}
\text{possible values}&\text{consequence of $x^2=1$ and periodicity}\\ \hline
a=q+kn,\ b=r+kn&2kn=p-r=-2\\
a=r+kn,\ b=q+kn&k=0,\quad (a,b,c)=(r,q,p)\\
a=q+kn,\ c=r+kn&k=0,\quad (a,b,c)=(q,p,r)\\
a=r+kn,\ c=q+kn&2kn=p-q=-1
\end{array}
\]
The first and last rows are impossible for $n\geq4$; the second
does not have $a$ as its middle value; and the third gives
$x(A)=A$.  Thus the left multiplier is $s_q$ when $x(A)\neq A$.
This proves all six rows of the first table.  In its four nontrivial
rows, the right operation is a Knuth move, and the left operation
is a dual Knuth move at the same actual position $p$ or $q$.

Finally, suppose $x(A)=A$.  The restriction $x|_A$ is one of the
four involutions of $A$.  Directly applying the same descent test
gives
\[
\begin{array}{c|c|c|c}
x|_A&\text{right multiplier}&\text{left multiplier}&y|_A\\ \hline
1&1&1&1\\
s_p&s_q&s_p&s_q\\
s_q&s_p&s_q&s_p\\
t_{p,r}&1&1&t_{p,r}
\end{array}
\]
Periodicity makes $A+n\ZZ$ and its complement invariant under $x$.  The multipliers in this table fix that
complement.  On $A+n\ZZ$, the last column is precisely the
conjugation of $x$ by $t_{p,r}$, proving the remaining case.
\end{proof}

For an affine fixed-point-free involution the last case of
Theorem~\ref{direct-prop} cannot occur: an involution of the
three-element set $A$ has a fixed point.

\begin{corollary}\label{invol-cor}
Let $\Phi(x)=(P,P,\rho)$ with $x\in\cF_n$.  For a nontrivial
transformation $x\simRBS{\widebar i}y$, let $j\in\{i-1,i\}$ be
the window position determined by Theorem~\ref{direct-prop}.
Let $T$ be the unique tabloid Knuth neighbor of $P$ which reverses
the membership of $\overline{i-1},\widebar i$ in the descent set.
Then $y=s_jxs_j$ and $\Phi(y)=(T,T,\rho)$.
Conversely, each nontrivial tabloid move $P\mapsto D_k(P)$ lifts to
an involutive transformation from $x$ to $\Psi(D_k(P),D_k(P),\rho)$.
\end{corollary}

\begin{proof}
The right star $x\mapsto xs_j$ reverses the two descent indicators.
By Theorem~\ref{Kthm} and Proposition~\ref{prop:descents} it
changes the recording tabloid to a Knuth neighbor with these
indicators reversed.  The tabloid is unique by the local covering
property in \cite[Lemma~3.22]{CLP}.  By Theorem~\ref{direct-prop}, the
subsequent left star gives the involution $y=s_jxs_j$.  Its
recording tabloid is still $T$, so its insertion tabloid is also
$T$ by Proposition~\ref{ambc-inverse-prop}.  The two tabloid moves
thus exchange the same pair of residues.  The changes to the raw
weight in Theorem~\ref{Kthm} and Proposition~\ref{dKprop} cancel,
as in the double-star formula of \cite[Proposition~5.10]{KP}.

For the converse, lift $P\mapsto D_k(P)$ to a right Knuth move
from $x$ by \cite[Proposition~3.23]{CLP}.  A neighboring entry
witnessing this Knuth move determines a rank-two star type.
Applying the left star of that type gives the required endpoint by
the forward assertion.  The lift specifies the window reflection,
which can differ from $s_k$.
\end{proof}

\begin{example}[A window move and its tabloid move]\label{ex:two-indices}
Let $x=[3,6,1,5,4,2]$.  Its entries at positions $3,4,5$ are
$1,5,4$.  Since $4$ lies between $1$ and $5$, the right star for
the pair $\{s_3,s_4\}$ interchanges positions $3$ and $4$.
Table~\ref{tab:two-indices} lists the insertion and recording
tabloids.  Ordinary row insertion computes them because all three
windows are permutations of $[6]$ \cite[Theorem~10.2]{CPY}.  On the recording
tabloid the entries exchanged are $\ol4$ and $\ol5$: indeed,
$T=D_4(P)$, whereas $\ol3$ and $\ol4$ belong to the same row of $P$.
\end{example}

\begin{table}[htbp]
\centering
\[
 P=\yttab{\ol1&\ol2\\\ol3&\ol4\\\ol5\\\ol6},
 \qquad
 T=\yttab{\ol1&\ol2\\\ol3&\ol5\\\ol4\\\ol6}.
\]
\begin{tabular}{c@{\qquad}c@{\qquad}c@{\qquad}c}
$w$ & window of $w$ & $P(w)$ & $Q(w)$\\ \hline
$x$        & $[3,6,1,5,4,2]$ & $P$ & $P$\\
$xs_3$     & $[3,6,5,1,4,2]$ & $P$ & $T$\\
$s_3xs_3$  & $[4,6,5,1,3,2]$ & $T$ & $T$
\end{tabular}
\caption{The two steps of a double star in period six.}
\label{tab:two-indices}
\end{table}

\section{AMBC shapes, weights and rotation}\label{shift-sect}

\subsection{The shapes and weights that occur}\label{shape-weights-sect}

For an involution $w$ of period $n$, write
$f(w)=|\{i\in[n]:w(i)=i\}|$.  By periodicity, this counts the
translation classes of fixed integers.  The following proposition
expresses $f(w)$ as the number of columns of odd length.

\begin{proposition}[Fixed points and odd columns]\label{fixed-point-shape}
Let $w\in\tS_n$ be an involution with AMBC shape $\lambda$.  Then
\[
  f(w)=\lambda_1-\lambda_2+\lambda_3-\cdots
      =|\{j:\lambda'_j\text{ is odd}\}|.
\]
\end{proposition}

\begin{proof}
Write $I_k=[1,kn]\cap\ZZ$, and form the word
$w(1),\ldots,w(kn)$.  To obtain a finite involution, keep the
positions $S_k=\{i\in I_k:w(i)\in I_k\}$.  If $i\in S_k$, then
$w(i)\in S_k$ because $w^2(i)=i$.  Hence $w|_{S_k}$ is a permutation
of $S_k$ with exactly $kf(w)$ fixed points.  Standardizing both its
domain and its values in increasing order preserves the fixed-point
count and its ordinary insertion shape.

Let $\alpha^{(k)}$ and $\beta^{(k)}$ be the ordinary insertion shapes
of the full word and the restricted word, respectively.  Set
$D=\max_{i\in[n]}|w(i)-i|$.  Any deleted position is among the first
or last $D$ positions of $I_k$, so at most $2D$ letters are deleted.
By Greene's theorem \cite{Greene}, the sum of the first $r$ row
lengths is the largest number of letters covered by $r$ disjoint
increasing subsequences (allowing empty subsequences).  Removing
the deleted letters from an optimal family therefore gives
\[
0\leq\sum_{i=1}^{r}(\alpha_i^{(k)}-\beta_i^{(k)})\leq 2D .
\]
Missing parts are read as zero.  The one-sided insertion theorem
\cite[Theorem~7.3]{CPY} gives
$\alpha_i^{(k+1)}-\alpha_i^{(k)}=\lambda_i$ for all sufficiently
large $k$, so $\alpha_i^{(k)}=k\lambda_i+O(1)$.  Subtracting two
successive partial sums in the bound above gives
$|\alpha_i^{(k)}-\beta_i^{(k)}|\leq 2D$.  Hence
$\beta_i^{(k)}/k\longrightarrow\lambda_i$.

A decreasing subsequence cannot contain two positions congruent
modulo $n$, since $w(i+tn)=w(i)+tn$ for $t>0$.  By Schensted's
theorem the two insertion shapes thus have at most $n$ rows.
Sch\"utzenberger's fixed-point theorem for the finite involution
$w|_{S_k}$ \cite[Theorem~4.1]{Beissinger} now gives
\[
kf(w)=\sum_{i=1}^{n}(-1)^{i-1}\beta_i^{(k)}.
\]
Dividing by $k$ and taking the limit gives the alternating sum of
the rows of $\lambda$.  The contributions within each column cancel
in pairs, leaving one exactly when the column has odd length.
\end{proof}

\begin{corollary}[AMBC image of the fixed-point-free involutions]
\label{fpf-image}
Let $n$ be even.  The image $\Phi(\cF_n)$ consists of all triples
$(P,P,\rho)$ with the following properties.
\begin{enumerate}
\item $P$ is any tabloid of a shape $\lambda\vdash n$ with even
      column lengths.
\item On every maximal block $[a,b]$ of equal parts of $\lambda$,
      the integer vector $\rho$ satisfies
      $\rho_a\leq\cdots\leq\rho_b$ and
      $\rho_i=-\rho_{a+b-i}$ for $a\leq i\leq b$.
\end{enumerate}
\end{corollary}

\begin{proof}
Inversion interchanges the two tabloids.  When they agree, the offset
is zero, and sorting $-\rho$ into dominant order reverses each block
of $-\rho$.  By Proposition~\ref{ambc-inverse-prop} and injectivity
of $\Phi$, the inverse image of a dominant diagonal triple is an
involution exactly when the stated block identities hold.
For arbitrary data satisfying (1)--(2), each block has sum zero,
so the triple belongs to $\dom$ and has a preimage in
$\tS_n$.  This preimage is an involution by the preceding argument.
Proposition~\ref{fixed-point-shape} then shows that it is fixed-point-free
exactly when all columns have even length.
\end{proof}

\begin{example}[A block of four equal rows]\label{ex:four-equal-rows}
At shape $(3,3,1,1,1,1)$ the column lengths are $6,2,2$.  The first
two coordinates of an occurring weight have the form $(-a,a)$.
On the last four rows, Corollary~\ref{fpf-image} gives
$(-c,-b,b,c)$ with $0\le b\le c$.  For example,
$(-1,1,-3,-2,2,3)$ occurs.  Replacing the last four coordinates by
$(-3,-1,2,2)$ preserves their sum and their weak order, but violates
$\rho_3=-\rho_6$.
\end{example}

Henceforth $\Pafpf(n)$ denotes these occurring shapes.  They are
equivalently the partitions obtained by writing each part of a
partition of $n/2$ twice.  For instance, the shape $(2,2,1,1)$ has
column lengths $(4,2)$ and occurs with every weight $(-a,a,-b,b)$
for $a,b\geq0$.  The finite involution $[3,2,1,4]$, extended with
period four, provides an example with two fixed residues.  Its finite insertion
shape $(2,1,1)$ is also its AMBC shape by \cite[Theorem~10.2]{CPY}.
Its two columns have lengths $3$ and $1$.

\subsection{Rotation at fixed shape and weight}
Fix a shape $\lambda$ and a weight $\rho$ satisfying
Corollary~\ref{fpf-image}. As $P$ ranges over the tabloids of
shape $\lambda$, the vertices with these data are precisely the
involutions $\Psi(P,P,\rho)$. Conjugation by $\omega$ takes the
vertex indexed by $P$ to the one indexed by $\omega(P)$
(Proposition~\ref{omega_prop}). We determine its action on
the descent labels and edge weights.

\begin{proposition}[Cyclic symmetry of the affine FPF $W$-graph]
\label{omega-wgraph-prop}
For $k\in\ZZ$, define
\[
\theta_k:\cF_n\longrightarrow\cF_n,
\qquad
\theta_k(x)=\omega^kx\omega^{-k}.
\]
Then $\theta_k$ is an $\alpha_k$-twisted automorphism of
$\Gamma_n^\m$.  More precisely, for all $x,y\in\cF_n$,
\[
\nu_\m(\theta_k(x))=\alpha_k(\nu_\m(x))
\qquad\text{and}\qquad
\omega_\m(\theta_k(x),\theta_k(y))=\omega_\m(x,y).
\]
Moreover, if $\Phi(x)=(P,P,\rho)$, then
\[
\Phi(\theta_k(x))=(\omega^k(P),\omega^k(P),\rho).
\]
\end{proposition}

\begin{proof}
Conjugation by $\omega^k$ induces the Coxeter-system automorphism
$\alpha_k(s_i)=s_{i+k}$ of $\tS_n$.  It preserves $\ellfpf$ and the
FPF Bruhat order.  Let the same symbol $\alpha_k$ denote the induced
automorphism of $\cH$, and define the $\cA$-linear bijection
\[
F_k:\cM\longrightarrow\cM,
\qquad F_k(M_x)=M_{\theta_k(x)}.
\]
The three cases in the standard-basis action in
Proposition~\ref{fpf-module-action} give
\[
F_k(HM)=\alpha_k(H)F_k(M)
\qquad(H\in\cH,\ M\in\cM).
\]
The map $\theta_k$ permutes the two minimal FPF involutions
$\Theta^+$ and $\Theta^-$.  The uniqueness of the bar operator in
Corollary~\ref{fpf-bar-canonical} consequently implies
\[
F_k(\overline M)=\overline{F_k(M)}.
\]
Since $\theta_k$ also preserves the FPF Bruhat order, the triangular
uniqueness of the canonical basis gives
\[
F_k(\underline M_y)=\underline M_{\theta_k(y)}.
\]
Comparing standard-basis coefficients yields
\[
\m_{x,y}=\m_{\theta_k(x),\theta_k(y)}
\qquad\text{and}\qquad
\mu_\m(x,y)=\mu_\m(\theta_k(x),\theta_k(y)).
\]
The equality for $\nu_\m$ follows directly from its definition, and
substitution in the definition of $\omega_\m$ proves the asserted
equality of all directed edge weights.  The AMBC formula follows by
iterating Proposition~\ref{omega_prop}.
\end{proof}

\begin{example}[Rotation across the end of a window]\label{ex:rotation-window}
Take $w=[2,1,4,3,6,5]$.  Row insertion gives top row
$\{\ol1,\ol3,\ol5\}$ and bottom row $\{\ol2,\ol4,\ol6\}$ in both
tabloids, with weight $(0,0)$.  Since
$(\omega w\omega^{-1})(j)=w(j-1)+1$, its rotated window is
$[0,3,2,5,4,7]$.  Proposition~\ref{omega-wgraph-prop} gives top row
$\{\ol2,\ol4,\ol6\}$ and bottom row $\{\ol1,\ol3,\ol5\}$ for this
window, again with weight $(0,0)$.  The descent labels rotate from
$\{s_1,s_3,s_5\}$ to $\{s_2,s_4,s_0\}$.  For the last descent one
compares the sixth entry $7$ to the seventh entry $0+6=6$.
\end{example}

\begin{proposition}\label{omega-it-prop}
For affine involutions $x,y$, one has
\[
x\simRBS{\widebar i}y
\quad\Longleftrightarrow\quad
\omega x\omega^{-1}
\simRBS{\overline{i+1}}
\omega y\omega^{-1}.
\]
\end{proposition}
\begin{proof}
Conjugation by $\omega$ sends $s_j$ to $s_{j+1}$ and shifts
both left and right descent sets by one.  It therefore carries
$\mathcal D_R(i)$ and $\mathcal D_L(i)$ to
$\mathcal D_R(i+1)$ and $\mathcal D_L(i+1)$, respectively, and
commutes with the corresponding extended rank-two star
operations.  Applying the same argument to $\omega^{-1}$ gives
the converse.
\end{proof}

\begin{proposition}\label{omega-prop}
If $w$ is an affine fixed-point-free involution, then $w$ and $\omega w\omega^{-1}$ cannot be connected by a sequence of involutive transformations.
\end{proposition}
\begin{proof}
Suppose $w$ belongs to the conjugacy class of $\Theta^\pm$.  Then $\omega w\omega^{-1}$ belongs to the conjugacy class of $\Theta^\mp$, since $\omega\Theta^+\omega^{-1}=\Theta^-$.  Involutive transformations preserve the conjugacy class, so $w$ and $\omega w\omega^{-1}$ cannot be connected by a sequence of such transformations.
\end{proof}

\begin{theorem}\label{omega-thm}
Suppose $w$ is an affine fixed-point-free involution with $\Phi(w)=(P,P,\rho)$. For each $P'$ of the same shape as $P$, let $w'=\Psi(P',P',\rho)$. Then there exists $i\in[n]$ such that $w$ and $\omega^i w'\omega^{-i}$ are connected by a sequence of involutive transformations.
\end{theorem}
\begin{proof}
Let $\lambda$ be the shape of $P$. Since $d_\lambda$ divides the sum
$n$ of the column lengths, we can choose $i\in[n]$ such that
\[
 i\equiv\operatorname{charge}(P')-\operatorname{charge}(P)
 \pmod{d_\lambda}.
\]
Rotation decreases charge by one modulo $d_\lambda$, so $P$ and
$\omega^i(P')$ have the same charge modulo $d_\lambda$. By the
component criterion in Section~\ref{pre:knuth}, there is a path of
tabloid Knuth moves $P=T_0,T_1,\ldots,T_r=\omega^i(P')$.
By Corollary~\ref{fpf-image}, each $z_k=\Psi(T_k,T_k,\rho)$ is
fixed-point-free. Corollary~\ref{invol-cor} lifts each step of the
path to an involutive transformation from $z_k$ to $z_{k+1}$.
The endpoints are $z_0=w$ and $z_r=\omega^i w'\omega^{-i}$ by
Proposition~\ref{omega_prop} and injectivity of $\Phi$.
\end{proof}

\section{Molecules for \texorpdfstring{$\Gamma^\m_n$}{Gamma(m)}}\label{mole-sect}
\subsection{The bidirected graph at a fixed shape and weight}
Corollary~\ref{fpf-image} describes the vertices of $\Gamma_n^\m$ at
each occurring shape and weight. We now read their descent labels
from the common tabloid and determine which pairs are joined by a
bidirected edge and what its two directed weights are.

\begin{proposition}[Descent labels in AMBC]
\label{fpf-descent-label-prop}
Let $x\in\cF_n$ and suppose that
$\Phi(x)=(P,P,\rho)$.  Under the identification
$s_i\leftrightarrow\widebar i$ between $S$ and $\ZZ/n\ZZ$, one has
\[
\nu_\m(x)=L(x)=R(x)=\tau(P).
\]
\end{proposition}

\begin{proof}
By the simple conjugation height formula of
Section~\ref{prelim-sect}, the height drops exactly when
$x(i)>x(i+1)$ and $x(i+1)\ne i$. It is unchanged if $x(i+1)=i$,
when $s_ixs_i=x$. Thus
\[
s_i\in\nu_\m(x)
\quad\Longleftrightarrow\quad
s_ixs_i\leq_Fx
\quad\Longleftrightarrow\quad
x(i)>x(i+1).
\]
Since $x=x^{-1}$, the last condition is equivalent to
$\widebar i\in L(x)$; it is also equivalent to
$\widebar i\in R(x)$.  Proposition~\ref{prop:descents} now gives
$L(x)=R(x)=\tau(P)$.
\end{proof}

\begin{proposition}[Rank-two criterion for bidirected edges]
\label{rank-two-edge-criterion}
Let $x<_{F}y$ in $\cF_n$.  There is a bidirected edge between
$x$ and $y$ in $\Gamma_n^{\m}$ if and only if there are
noncommuting simple reflections $s,t\in S$ such that
\[
txt\leq_F x<_F sxs=y<_{F}tyt.
\]
The noncommuting pairs are precisely $\{s_{i-1},s_i\}$ for $i\in\ZZ$.
Whenever these conditions hold,
\[
\omega_\m(x,y)=\omega_\m(y,x)=1.
\]
\end{proposition}

\begin{proof}
Suppose first that the edge is bidirected.  Then
$\nu_\m(x)$ and $\nu_\m(y)$ are incomparable.  Choose
\[
s\in\nu_\m(y)\setminus\nu_\m(x),
\qquad
t\in\nu_\m(x)\setminus\nu_\m(y).
\]
Thus $sys\leq_Fy$, $sxs>_Fx$, $txt\leq_Fx$, and
$tyt>_Fy$.  Since $x<_{F}y$, triangularity gives
$\mu_\m(y,x)=0$.  Lemma~3.25(a) of \cite{Marberg}, applied
with $s$, now gives
\[
0\neq\mu_\m(x,y)=\delta_{sxs,y},
\]
so $y=sxs$ and the displayed chain follows.

The reflections $s$ and $t$ cannot commute.  If they commuted, then
$tyt=s(txt)s$.  Since conjugation by a simple reflection changes
the FPF length by at most one, this would give
\[
\ellfpf(tyt)\leq\ellfpf(txt)+1
\leq\ellfpf(x)+1=\ellfpf(y),
\]
contrary to $tyt>_Fy$.  In affine type $A_{n-1}$ with $n\geq4$,
two distinct noncommuting simple reflections are a cyclically
adjacent pair.

Conversely, the displayed chain implies that
\[
s\in\nu_\m(y)\setminus\nu_\m(x),
\qquad
t\in\nu_\m(x)\setminus\nu_\m(y).
\]
Lemma~3.25(a) of
\cite{Marberg} gives $\mu_\m(x,y)=1$, so both directed edge
weights are nonzero and the edge is bidirected.
In either direction, triangularity gives $\mu_\m(y,x)=0$.
Together with the incomparability of the two descent labels,
the definition of $\omega_\m$ therefore gives
$\omega_\m(x,y)=\omega_\m(y,x)=1$.
\end{proof}

\begin{proposition}\label{bi-edge-it-prop}
Distinct elements $x,y\in\cF_n$ are joined by a bidirected edge
in $\Gamma_n^{\m}$ if and only if
\[
x\simRBS{\widebar i}y
\]
for some $i\in\ZZ$.
\end{proposition}

\begin{proof}
Suppose first that $x\simRBS{\widebar i}y$ with $x\neq y$, and
let $\{s,t\}=\{s_{i-1},s_i\}$.  An involution preserving
$\{i-1,i,i+1\}$ would have a fixed point in this odd set.
Thus Theorem~\ref{direct-prop} shows that, for $x\in\cF_n$,
the right star and the subsequent left star use the same reflection,
say $s$, and $y=sxs$.
The right star interchanges the unique member of $\{s,t\}$ in
the right descent set.  The subsequent dual Knuth move preserves
the right descent set, since it preserves the recording tabloid
by Proposition~\ref{dKprop}.  Using Proposition~\ref{prop:descents}
and the identification $s_j\leftrightarrow\widebar j$, we may
therefore interchange $x$ and $y$ so that
\[
R(x)\cap\{s,t\}=\{t\},
\qquad
R(y)\cap\{s,t\}=\{s\},
\qquad
y=sxs.
\]
The FPF length formula of \cite[Corollary~3.16]{Zhang} gives
\[
txt\leq_Fx<_F sxs=y<_Ftyt.
\]
Proposition~\ref{rank-two-edge-criterion} now gives a
bidirected edge.

Conversely, suppose that $x$ and $y$ are joined by a bidirected
edge.  By triangularity, $\mu_\m(x,y)$ and $\mu_\m(y,x)$
both vanish when $x$ and $y$ are incomparable.  We may therefore
interchange the endpoints and assume that $x<_Fy$.
Proposition~\ref{rank-two-edge-criterion} gives an adjacent pair
$\{s,t\}=\{s_{i-1},s_i\}$ with
\[
txt\leq_Fx<_F sxs=y<_Ftyt.
\]
By Proposition~\ref{fpf-descent-label-prop}, the right descent
sets have exactly the two patterns displayed above.
We check that the right star uses $s$.
Put $p=i-1$, $q=i$, $r=i+1$, and
$(a,b,c)=(x(p),x(q),x(r))$.

If $s=s_p$, the descent pattern of $x$ gives $a<b$ and $c<b$.
The right star uses $s_p$ when $a<c<b$, so the only other
possibility is $c<a<b$.  But $s_q\notin R(y)$ gives
\[
s_p(a)=y(q)<y(r)=s_p(c).
\]
Since $c<a$, this reversal of order under $s_p$ forces
$a=q+kn$ and $c=p+kn$ for some $k\in\ZZ$.
The involution condition and periodicity then give
\[
r=x(c)=x(p+kn)=a+kn=q+2kn,
\]
which would require $2kn=1$.

If $s=s_q$, the descent pattern gives $b<a$ and $b<c$.
The right star uses $s_q$ when $b<a<c$.  The other possibility,
$b<c<a$, together with $s_p\notin R(y)$, would give
\[
s_q(a)=y(p)<y(q)=s_q(c).
\]
This forces $a=r+kn$ and $c=q+kn$, whence
\[
p=x(a)=x(r+kn)=c+kn=q+2kn,
\]
or $2kn=-1$, again impossible.

Thus the right star uses $s$ in either case.  Theorem~\ref{direct-prop}
then shows that the left star also uses $s$, so their composite
sends $x$ to $sxs=y$.  Hence $x\simRBS{\widebar i}y$.
\end{proof}

Recall that $\mathcal A_\lambda$ is the tabloid Knuth graph of shape
$\lambda$ defined in Section~\ref{pre:knuth}.

\begin{proposition}[AMBC identification of the bidirected graph]
\label{ambc-bidirected-graph-prop}
Fix a shape $\lambda$ and an integer weight $\rho$ satisfying
Corollary~\ref{fpf-image}, and set
\[
V_{\lambda,\rho}
=
\{x\in\cF_n:\Phi(x)=(P,P,\rho),\
  \operatorname{shape}(P)=\lambda\}.
\]
Equip $\mathcal A_\lambda$ with the vertex label
\[
T\longmapsto\{s_i:\widebar i\in\tau(T)\}
\]
and put weight one on both directions of every edge.  Then the map
\[
V_{\lambda,\rho}\longrightarrow\RSYT(\lambda),
\qquad
x\longmapsto P(x),
\]
is an isomorphism from the descent-labeled, weighted bidirected graph
induced by $\Gamma_n^{\m}$ on $V_{\lambda,\rho}$ to
$\mathcal A_\lambda$ with this additional structure.
\end{proposition}

\begin{proof}
By Corollary~\ref{fpf-image}, every tabloid $T$ of shape $\lambda$
occurs in the fiber, and AMBC injectivity gives exactly one vertex
for each $T$.

If $x$ and $y$ are bidirected neighbors, Proposition~\ref{bi-edge-it-prop}
and Corollary~\ref{invol-cor} show that their common tabloids are
$P$ and $D_j(P)$ for some tabloid swap index $j$.  Conversely,
if $D_j(P)\neq P$, the converse part of
Corollary~\ref{invol-cor} gives a rank-two involutive
transformation between $\Psi(P,P,\rho)$ and
$\Psi(D_j(P),D_j(P),\rho)$; Proposition~\ref{bi-edge-it-prop}
then gives a bidirected edge.  Thus adjacency is preserved in
both directions.  To complete the identification, apply
Proposition~\ref{fpf-descent-label-prop} to the vertex labels and
Proposition~\ref{rank-two-edge-criterion} to the two directed weights
on each edge, which are both one.
\end{proof}

\subsection{Molecule membership and cross-weight symmetry}
\begin{theorem}\label{mol-m-thm}
Two elements $w,v\in\cF_n$ lie in the same molecule of $\Gamma^\m_n$ if and only if they are connected by a sequence of involutive transformations.
\end{theorem}
\begin{proof}
This follows from Proposition~\ref{bi-edge-it-prop}.
\end{proof}

\begin{corollary}
If two affine fixed-point-free involutions $w$ and $v$ lie in the same molecule of $\Gamma^\m_n$, then they have the same FPF sign and their AMBC images have the same shape and dominant weight.
\end{corollary}
For $n=4$, consider $w=[4,3,2,1]$ and its rotation
$w'=\omega^2w\omega^{-2}=[0,-1,6,5]$. Applying the transformations
of Theorem~\ref{direct-prop} to $w$ gives the first of the following
sets, and rotating that set by two gives the second
\cite[Section~6]{ZhangFPSAC}.
\[
\{[4,3,2,1],[-4,3,2,9],[3,-4,1,10],[-5,4,9,2],[4,-5,10,1],[4,11,-6,1]\}
\]
and 
\[
\{[0,-1,6,5],[0,7,-2,5],[7,0,-3,6],[-1,8,5,-2],[8,-1,6,-3],[-8,-1,6,13]\}.
\]
Both representatives have FPF sign $1$.  For $w$, the balls in one
period form the antidiagonal of the $4\times4$ square.  The first
AMBC step removes the stream $(1,1)+\ZZ(4,4)$ and leaves the
antidiagonal on rows and columns $2,3,4$.  Repeating the step removes
$(j,j)+\ZZ(4,4)$ for $j=2,3,4$.  Thus the four singleton rows of
$P(w)$ read $1,2,3,4$ from top to bottom and all four stream
altitudes are zero.  Proposition~\ref{omega_prop} gives the same
weight for $w'$ and the rows $3,4,1,2$.  By the charge formula in Section~\ref{pre:knuth}, the charges of these two tabloids are $0$ and $2$, distinct
modulo $d_{(1,1,1,1)}=4$.  The two representatives therefore belong
to different molecules with the same sign, shape, and weight.

\begin{theorem}[Cross-weight symmetry of the bidirected parts]
\label{omega1-thm}
Let $M$ and $M'$ be two molecules of $\Gamma^\m_n$ with the same
AMBC shape $\lambda$, with arbitrary AMBC weights.  Then there are
$k\in\ZZ$ and an $\alpha_k$-twisted isomorphism from the
descent-labeled, weighted bidirected graph on $M$ to the corresponding
graph on $M'$.  In particular, the isomorphism preserves both directed
weights on every retained edge, and all these weights are equal to one.
\end{theorem}
\begin{proof}
Let $\rho$ and $\rho'$ be the AMBC weights of $M$ and $M'$, and let
$C$ and $C'$ be the corresponding connected components of
$\mathcal A_\lambda$ under
Proposition~\ref{ambc-bidirected-graph-prop}.  By
\cite[Remark~8.8]{CLP}, cyclic rotation permutes the components of
$\mathcal A_\lambda$ transitively.  Choose $k\in\ZZ$ such that
$C'=\omega^k(C)$.  The rule
\[
\Psi(T,T,\rho)\longmapsto
\Psi(\omega^k(T),\omega^k(T),\rho')
\qquad(T\in C)
\]
is a bijection from $M$ to $M'$.  Proposition~\ref{ambc-bidirected-graph-prop}
shows that it preserves bidirected adjacency and both directed edge
weights.  Moreover,
\[
\{s_i:\widebar i\in\tau(\omega^k(T))\}
=
\alpha_k\bigl(\{s_i:\widebar i\in\tau(T)\}\bigr),
\]
so it rotates the descent labels by $\alpha_k$.  This proves the
result for arbitrary occurring weights $\rho$ and $\rho'$.
\end{proof}

\subsection{The number and size of the molecules}\label{size-sect}

\begin{proposition}\label{molecule-number-prop}
Let $(\lambda,\rho)$ satisfy the conditions of
Corollary~\ref{fpf-image}.  The number of molecules of $\Gamma_n^{\mathrm m}$ having
AMBC shape $\lambda$ and dominant weight $\rho$ is $d_\lambda$.

Equivalently, for any tabloid $P$ of shape $\lambda$, the least
positive integer $m$ such that $P$ and $\omega^m(P)$ are in the
same connected component of $\mathcal A_\lambda$ is
$d_\lambda$.
\end{proposition}

\begin{proof}
Proposition~\ref{ambc-bidirected-graph-prop} identifies the simple
bidirected graph on $V_{\lambda,\rho}$ with
$\mathcal A_\lambda$.  The molecules in this fiber are therefore
indexed by the connected components of $\mathcal A_\lambda$.

By the charge criterion of Section~\ref{pre:knuth}, two
tabloids give the same component precisely when their charges
agree modulo $d_\lambda$. Each rotation decreases the charge by one,
so successive rotations realize all $d_\lambda$ residues. In particular
$T$ and $\omega^m(T)$ lie in the same component precisely when
$d_\lambda\mid m$.
Thus the least positive such $m$ is $d_\lambda$.
\end{proof}

\begin{corollary}[Size and sign of the components]\label{size-sign-cor}
Fix any shape and weight in Corollary~\ref{fpf-image}.
Every molecule in this fiber has
\[
\frac{n!}{d_\lambda\prod_{i=1}^{\ell(\lambda)}\lambda_i!}
\]
vertices.  Exactly $d_\lambda/2$ molecules have each FPF sign.
\end{corollary}

\begin{proof}
Choosing unordered entries for each ordered row of a tabloid gives
$\mathcal A_\lambda$ a total of $n!/\prod_i\lambda_i!$ vertices.  Rotation permutes its $d_\lambda$
components transitively, so each has the same size.
Proposition~\ref{ambc-bidirected-graph-prop} carries this count to
the molecules.  One rotation switches the FPF sign while preserving
the weight, and the sign is constant on a molecule.  The signs
therefore alternate around the cycle of components.
\end{proof}

In particular, $2\mid d_\lambda\mid n$, since the column lengths
are even and sum to $n$.  For $\lambda=(1^n)$ this gives
$n$ molecules of $(n-1)!$ vertices each.  For $\lambda=(m,m)$,
where $n=2m$, it gives two molecules of $\tfrac12\binom{2m}{m}$
vertices each at every weight $(-a,a)$, $a\geq0$.

\begin{example}[Periods four and six]\label{small-molecules}
Table~\ref{tab:small-molecules} lists all shapes for $n=4,6$.
Each row gives the counts at any one occurring weight.  For $(2,2)$,
the top row determines the entire tabloid.
Figure~\ref{fig:four-molecules} identifies the two components using
these top rows.  In its upper path the descent labels from left to
right are $\{s_1\}$, $\{s_2,s_4\}$, and $\{s_3\}$, with $s_4=s_0$.
Thus each horizontal swap interchanges adjacent residues and has
incomparable descent sets at its endpoints.  The diagram is the same
at $(0,0)$ and $(-1,1)$: matching equal top rows gives the
isomorphism in Theorem~\ref{omega1-thm} with no rotation of labels.
\end{example}

\begin{table}[htbp]
\centering
$\begin{array}{c|c|r|r|r}
n&\lambda&|\RSYT(\lambda)|&d_\lambda&\text{vertices per molecule}\\ \hline
4&(2,2)&6&2&3\\
4&(1^4)&24&4&6\\ \hline
6&(3,3)&20&2&10\\
6&(2,2,1,1)&180&2&90\\
6&(1^6)&720&6&120
\end{array}$
\caption{Tabloid counts, numbers of molecules, and their sizes at a fixed weight.}
\label{tab:small-molecules}
\end{table}

\begin{figure}[htbp]
\centering
$\begin{array}{c@{\qquad}ccccc}
\text{charge }1 & \{1,4\} & \overset{D_1}{\longleftrightarrow} &
 \{2,4\} & \overset{D_3}{\longleftrightarrow} & \{2,3\}\\
 &\big\downarrow &&\big\downarrow &&\big\downarrow\\
\text{charge }0 & \{1,2\} & \overset{D_2}{\longleftrightarrow} &
 \{1,3\} & \overset{D_4}{\longleftrightarrow} & \{3,4\}
\end{array}$
\caption{Molecules at shape $(2,2)$.  Horizontal arrows are
bidirected edges.  Each vertical arrow adds one to every residue
modulo four.}
\label{fig:four-molecules}
\end{figure}

\subsection{Minimal elements in a molecule}\label{sec:minimal}

All comparisons in a molecule use the restriction of $\le_F$.
A vertex is \emph{minimal} if no other vertex of the molecule lies
strictly below it, and it is \emph{least} if it lies below every vertex.
In a finite poset, a unique minimal element is least.
Proposition~6.6 of \cite{ZhangFPSAC} asserts that every molecule has
a unique minimal element. Theorem~\ref{rectangular-minimum-thm}
shows that this holds precisely for rectangular shapes.
Example~\ref{ex:zero-minima} gives a molecule of shape $(2,2,1,1)$
with three minimal elements.

At a fixed shape and weight, the number of Bruhat-minimal elements
is independent of the charge class.  Indeed, conjugation by $\omega$
preserves the FPF Bruhat order and the AMBC weight
(Proposition~\ref{omega-wgraph-prop}), and sends the component of
charge $c$ to that of charge $c-1$ modulo $d_\lambda$
(Proposition~\ref{molecule-number-prop}).  Its powers therefore
identify the induced orders on any two molecules in the same fiber.

\begin{proposition}
\label{minimal-orbit-prop}
Let $\Gamma$ be a molecule of $\Gamma_n^{\mathrm m}$ of
AMBC shape $\lambda\in\Pafpf(n)$,
and let
\[
d_\lambda=\gcd(\lambda'_1,\lambda'_2,\ldots).
\]
Then the set of elements of $\Gamma$ that are minimal with respect to
the FPF Bruhat order $\leq_F$ is invariant under
the map
\[
z\longmapsto \omega^{d_\lambda}z\omega^{-d_\lambda}.
\]
In particular, if $\Gamma$ has a unique minimal element, then
\[
\frac{n}{d_\lambda}\mid\lambda_i
\qquad\text{for every part $\lambda_i$ of $\lambda$.}
\]
\end{proposition}

\begin{proof}
By Proposition~\ref{molecule-number-prop}, conjugation by
$\omega^{d_\lambda}$ preserves every molecule of AMBC shape
$\lambda$. This conjugation is induced by an automorphism of the
affine Coxeter diagram, and therefore preserves the FPF Bruhat
order. Hence it permutes the minimal elements of $\Gamma$.

Suppose that $\Gamma$ has a unique minimal element $z$, and
write
\[
\Phi(z)=(P,P,\rho).
\]
The uniqueness of $z$ implies
\[
\omega^{d_\lambda}z\omega^{-d_\lambda}=z.
\]
Proposition~\ref{omega_prop} and the injectivity of $\Phi$ give
\[
\omega^{d_\lambda}(P)=P.
\]

The permutation of $\mathbb Z/n\mathbb Z$ given by
$x\mapsto x+d_\lambda$ has orbits of size
\[
\frac{n}{d_\lambda}.
\]
Each row of $P$ is invariant under this permutation, so it is a union
of such orbits. Consequently,
\[
\frac{n}{d_\lambda}\mid |P_i|=\lambda_i
\]
for every row $P_i$.
\end{proof}

\begin{corollary}\label{minimal-multiplicity}
For a molecule of shape $\lambda$, put $m=n/d_\lambda$ and
$g=\gcd(m,\lambda_1,\ldots,\lambda_\ell)$.  Its number of minimal
elements in the FPF Bruhat order is divisible by $m/g$.
\end{corollary}

\begin{proof}
By Proposition~\ref{minimal-orbit-prop}, conjugation by
$\omega^{d_\lambda}$ permutes the minimal elements.  This action has
order dividing $m$, since translation by $n$ commutes with every
affine permutation.  Let $t$ be the length of an orbit.  Then $t\mid m$,
and the tabloid at each vertex of the orbit is fixed by adding
$t d_\lambda$ modulo $n$.  This addition has residue orbits of size
$m/t$.  Every tabloid row is a union of such orbits, so $m/t$ divides
every $\lambda_i$ and hence divides $g$.  Thus $m/g$ divides $t$.
Summing over the orbits of minimal elements proves the assertion.
\end{proof}

\begin{theorem}[Rectangular molecules]\label{rectangular-minimum-thm}
Let $\Gamma$ be a molecule of $\Gamma_n^{\m}$ with AMBC shape $\lambda$
and weight $\rho$.  If $\lambda$ is not a rectangle, then $\Gamma$ has
at least three minimal elements in the FPF Bruhat order.
If $\lambda=(k^r)$ is a rectangle, then $\Gamma$ has a least element.
More precisely, let $c$ be the charge of $\Gamma$ modulo $r$, and define
\[
 u(i+rj)=r+1-i+r(j-\rho_i)
 \qquad(1\le i\le r,\ j\in\ZZ).
\]
The least element of $\Gamma$ is $\omega^K u\omega^{-K}$ for
$K\equiv-c\pmod r$.  In particular, a molecule has a unique minimal
element if and only if its shape is rectangular.
\end{theorem}

\begin{proof}
Put $d=d_\lambda$ and $e=\gcd_i\lambda_i$.  The multiplicity of each
distinct row length is divisible by $d$, so $e$ divides $n/d$.
By Corollary~\ref{minimal-multiplicity}, the number of minimal elements
is divisible by $n/(de)$.  If the distinct row lengths are
$b_1,\ldots,b_t$, with multiplicities $m_1,\ldots,m_t$, then
\[
  \frac{n}{de}=\sum_{i=1}^t\frac{m_i}{d}\frac{b_i}{e}.
\]
When $t\ge2$, the numbers $b_i/e$ include two distinct positive integers,
so this sum is at least three.

Now take $\lambda=(k^r)$, so $n=kr$.  By Corollary~\ref{fpf-image},
$r$ is even, $\rho$ is weakly increasing, and $\rho_{r+1-i}=-\rho_i$.
In particular $u$ is an involution.  It pairs distinct residues
modulo $r$, and $\sum_{i=1}^r u(i)=\sum_{i=1}^r i$ since $\sum_i\rho_i=0$.
Thus $u\in\cF_n$.

Let $P$ be a tabloid of shape $\lambda$.  List the entries in row $i$
in increasing order as $L_i(0),\ldots,L_i(k-1)\in[n]$, and extend each
list by $L_i(j+k)=L_i(j)+n$.  Write $t_i(j)=L_i(j+o_i)$.  Given $o_i$,
take the smallest $o_{i+1}$ such that $t_i(j)<t_{i+1}(j)$ for every $j$.
By the definition of local charge, $o_{i+1}-o_i=\lch_i(P)$.
Minimality of the shift gives
\[
  t_{i+1}(j-1)<t_i(j)\quad\text{for some $j$, for each $i<r$}.
  \tag{*}\label{rect-alignment-witness}
\]
Set $f(i+rj)=t_i(j)$, an element of $\widehat S_n$.  Its mean
displacement $K=n^{-1}\sum_{p=1}^n(f(p)-p)$ equals $\sum_i o_i$,
since shifting a row list by one adds
$n$ to the sum of its $k$ entries.  Hence
\[
 K=r o_1+\sum_{i=1}^{r-1}(r-i)\lch_i(P)
   \equiv-\operatorname{charge}(P)\pmod r .
\]
Shift all the $o_i$ by the same integer to arrange $0\le K<r$.

We claim that $v=fuf^{-1}$ has AMBC image $(P,P,\rho)$.
Put $b_i=-\rho_i$.  The balls of $v$, in row and column coordinates, are
\[
   \bigl(t_i(j),t_{r+1-i}(j+b_i)\bigr)
       \qquad(1\le i\le r,\ j\in\ZZ).
\]
Give each displayed ball the label $j$.  In each block $i$ both
coordinates increase with $j$.  At fixed $j$ the first coordinates
increase and the second decrease with $i$, since $b_i\ge b_{i+1}$.
Along a strictly southeast step from block $i$ at label $j$
to block $i'$ at label $j'$, suppose that $j'\le j$. An increase
in the row coordinate then forces $i'>i$, which makes the column
coordinate smaller. Thus labels strictly increase along the path.
Each block supplies $k$ balls per period.  Since the
labels increase by $k$ under translation by $(n,n)$, no periodic
chain can contain more balls per period.  Every channel therefore contains
one ball of each label.  The balls in the last block are weakly southwest
of all other balls with the same label, so they form the southwest channel.

These labels give the channel numbering of \cite[Definition~3.8]{CPY}.
Indeed, every ball can be reached from the last block along a strictly
southeast path on which labels increase by one at each step.  To pass
from block $i+1$ to block $i$, take labels $j-1,j$ as in
\eqref{rect-alignment-witness}.  The column also increases at this
step, by the reverse block order and $b_i\ge b_{i+1}$.  Translates
of this step, followed by successive balls within block $i$, reach every
label in that block.  Repeating this argument from the last block reaches
every block.  Such a path attains the supremum defining the channel
numbering because strict increase bounds the contribution of every other path.

The forward step of \cite[Definition~3.26]{CPY} forms its stream point
from the first row and last column of the balls with label $j$.  This
gives $(t_1(j),t_1(j+b_r))$, with both residue sets $P_1$ and altitude
$b_r$ (the shift in the increasing list for that row).
The remaining points at this step are
$(t_{i+1}(j),t_{r+1-i}(j+b_i))$ for $1\le i<r$, obtained by taking
the next row and the preceding column.  This is the same array with
row lists $t_2,\ldots,t_r$ and shifts $b_1,\ldots,b_{r-1}$.
Iteration gives residue sets $P_i$ on both sides and altitude
$b_{r+1-i}=\rho_i$ at step $i$, proving the claim.

We now descend from $v$ to the proposed least element.
If $f$ is not a translation, choose $l$ with $f^{-1}(l)>f^{-1}(l+1)$.
Write $l=t_i(j)$ and $l+1=t_{i'}(j')$.  The inequality
$i+rj>i'+rj'$ and the increase along each row and column of the
array $t$ force $j>j'$ and $i<i'$.  Left multiplication by $s_l$
preserves each inequality $f(p)<f(q)$ for $p<q$, since the only value
pairs whose order changes are $l+an,l+1+an$, already inverted in $f$.
In particular, the transformed lists remain increasing, their columns
remain increasing, and their witnesses in \eqref{rect-alignment-witness}
persist because $i+1+r(j-1)<i+rj$.  Thus $s_l f$ is the normalized
lift of the transformed tabloid.  Its mean displacement is still $K$,
since the two changed residue classes contribute opposite increments.
The AMBC calculation and Proposition~\ref{molecule-number-prop} now show
that $s_lvs_l$ belongs to the same molecule as $v$.

In the target array, $r+1-i>r+1-i'$ and $j+b_i>j'+b_{i'}$.
The array entry $t_{r+1-i}(j'+b_{i'})$ lies strictly between
$v(l+1)$ and $v(l)$. Thus $v(l)\ge v(l+1)+2$. Left multiplication
by $s_l$ preserves the order of two values this far apart, so $v$
and $s_l v$ both have a right descent at $l$.  As $v=v^{-1}$, it follows
that $\ell(s_lvs_l)=\ell(v)-2$, hence $s_lvs_l<_F v$.
The periodic displacement $f(p)-p$ is bounded, so $f$ has only
finitely many inversion classes. Each step removes one such class,
since $l$ is a left descent. Eventually $f^{-1}$ is increasing, so
$f$ is the
translation $\omega^K$.  By the claim and the injectivity of AMBC,
every vertex of the specified charge and weight has the form $fuf^{-1}$
above.  Its descending chain ends at $\omega^K u\omega^{-K}$.
Thus the endpoint is a least element.  Since $u$ commutes with $\omega^r$,
the formula is unchanged when the charge representative changes by a multiple of $r$.
\end{proof}

For instance, take shape $(2,2)$ and weight $(-1,1)$ in
Example~\ref{small-molecules}.  The charge-zero molecule has least
element $[4,-1,6,1]$, and its rotation $[-2,5,0,7]$ is the least element
of the charge-one molecule.  Their top rows are $\{1,3\}$ and
$\{2,4\}$, respectively, the middle vertices in Figure~\ref{fig:four-molecules}.

\begin{example}[Three minimal elements at weight zero]
\label{ex:zero-minima}
At shape $(2,2,1,1)$ and weight zero there are two molecules, each
with $90$ vertices and exactly three minimal elements in the FPF Bruhat
order. In the molecule containing $[3,6,1,5,4,2]$, the minimal elements are
\[
 [0,-1,4,3,8,7],\qquad [2,1,6,5,4,3],\qquad [4,3,2,1,6,5].
\]
All three have FPF length two. Rotation by one gives the minimal elements
of the other molecule. Thus the lower bound in
Theorem~\ref{rectangular-minimum-thm} is sharp.
\end{example}

\end{document}